\documentclass{article}
\usepackage[utf8x]{inputenc}
\usepackage{amsmath, amssymb, amsthm, mathrsfs,cleveref}
\usepackage{bbold}
\usepackage{tikz-cd}
\usepackage[margin=3cm]{geometry}
\usepackage{enumitem}

\newtheorem{thm}{Theorem}[section] \newtheorem*{theorem*}{Theorem}
\newtheorem{lemma}[thm]{Lemma} \newtheorem*{lemma*}{Lemma}

\newtheorem{corollary}[thm]{Corollary} \newtheorem*{corollary*}{Corollary}
\newtheorem{prop}[thm]{Proposition} \newtheorem*{proposition*}{Proposition}

\newtheorem*{definition*}{Definition}
\newtheorem{rmk}[thm]{Remark}

\newcommand{\Z}{\mathbb Z}
\newcommand{\N}{\mathbb N}

\newcommand{\R}{\mathbb R}

\newcommand{\E}{\mathbb E}

\newcommand\cA{{\mathcal A}}

\newcommand\cI{{\mathcal I}}

\newcommand\cL{{\mathcal L}}

\newcommand\eps{\varepsilon}

\newcommand\sgn{\operatorname{sgn}}
\newcommand\mic{\mathcal M}
\newcommand{\1}{{\mathbb 1}}

\DeclareMathOperator{\Leb}{Leb}
\DeclareMathOperator{\Var}{Var}

\newcommand{\banach}\cL
\def\thetaout{\theta^{out}}
\def\rout{r^{out}}
\def\thetain{\theta^{in}}
\def\rin{r^{in}}

\title{On some random billiards in a tube with superdiffusion}
\iftrue
\author{Henk Bruin
	\thanks{Faculty of Mathematics, University of Vienna, 
		Oskar Morgensternplatz 1, 1090 Vienna, Austria; {\it henk.bruin@univie.ac.at}}
	\and Niels Kolenbrander
	\thanks{Mathematisch Instituut,
		University of Leiden,
		Einsteinweg 55,
		2333 CC Leiden, The Netherlands;
		{\it niels.kolenbrander@outlook.com}}
	\and Dalia Terhesiu
	\thanks{Mathematisch Instituut,
		University of Leiden,
		Einsteinweg 55,
		2333 CC Leiden, The Netherlands;
		{\it daliaterhesiu@gmail.com}}
}
\fi
\date{\today}

\begin{document}
	
	\maketitle
	
	\abstract{\em We consider a class of random billiards in a tube, 
		where reflection angles at collisions with the boundary of the tube are
		random variables rather than deterministic (and elastic) quantities.
		We obtain a (non-standard) Central Limit Theorem for the displacement of a particle, which marginally fails to have a second moment
		w.r.t.\ the invariant measure of the random billiard.
		\\[3mm] {\it 2000 Mathematics Subject Classification.} Primary 37D50, Secondary 37A30, 60D05
		\\[1mm] {\it Keywords and phrases.} billiards, random billiards, limit law, non-Gaussian Central Limit Theorem
	}
	\section{Introduction}\label{sec:intro}	
	In this paper we prove a non-standard Central Limit Theorem for the horizontal displacement of a particle that moves horizontally, but in a random way, in an infinite strip (called the tube) of width $W$ that has a non-smooth boundary. A billiard system like this was first considered in~\cite{FY}, where the aim was to devise a more mathematical approach to studying a gas flow in a tube where collisions with the wall are not deterministic, but random. This randomness can be caused by some energy at the boundary, or by the presence of a chemical structure on the boundary. In the second case a question, already posed by Knudsen in~\cite{Knud}, is how the chemical structure affects the movement of the particles.
	
	The long-term behaviour of particles, can be understood by proving statistical properties about a corresponding billiard model such as introduced in~\cite{FY}. A recurring property of these models is that the post-collision angle will be random variables depending on the pre-collision angles. This gives rise to a Markov chain, of which variations were studied in~\cite{FE,FY,FZ10,FZ12}. Similar to what is done in the previously mentioned papers, we will model the roughness of the boundary of the tube by covering it with tiny so-called \textit{microstructures}, which are small billiards tables bounded by finitely many convex smooth curves and an open side at the boundary of the tube, see Section~\ref{sec:basicDef} for a precise description of these microstructures. We will compute the trajectories in these microstructures using the rules of deterministic, elastic billiards. The randomness is sitting only in the entrance point of the particle into a microstructure; this will be a random variable on the open side. This is motivated by the fact that microstructures model microscopic chemical structures which are tiny compared to the width $W$ of the tube. The randomness captures the effect that from the macroscopic scale of the tube it is hard to predict where exactly a microscopic structure will be entered, so that it will seem random.
	
	The geometric setup is somewhat analogous to the Lorentz gas in a one-dimensional tube with infinite horizon, see \cite{SV07}, especially since we model the microstructures to be tangent to the tube. Several refined stochastic properties of the deterministic Lorentz gas
	with infinite horizon (for dimensions $1$ and $2$) have been obtained in recent years, see~\cite{ChDo09,BBT,BT25,PT23} and references therein. See also the recent survey~\cite{Pene25}.
	Substantial progress was made in understanding a certain class of random perturbations (for instance, random perturbations of the shape of the scatterers) of finite and infinite horizon billiards~\cite{DZ13}. In the finite horizon case, a (standard) central limit theorem is also obtained in~\cite{DZ13}.
	
	In the random billiard model considered here we don't prescribe the exact positions
	of the particle along the boundary of the tube and as a consequence, there is no spacial periodicity. Therefore, the existing  methods of  treating determistic Lorentz gases do not work as such.
	We consider the position of the particle once it reaches the boundary of the tube as random, and use this randomization to replace the two-dimensional billiard map by a piecewise expanding one-dimensional random map $\Psi_{R_i}$. The parameter $R_i$ is a random variable taking values in $\left[0,1\right]$, distributed according to some probability measure $\nu$ and representing the entry point into the microstructure when the particle reaches the boundary for the $i$-th time, see Sections~\ref{sec:basicDef}  and~\ref{sec:rtr} for a precise description. Similar to the deterministic billiard case, where the invariant measure is of the form $d \bar{\mu}(r,\theta) = C\sin(\theta)\,dr\, d\theta$, the maps $\Psi_{R_i}$ preserve a measure $d\mu(\theta) = \frac{1}{2}\sin(\theta) d \theta$, for each $R_i$ (it is in fact the restriction of $\bar{\mu}$ to the angle component). Instead of modelling the random billiard system by a Markov chain, we carry out the analysis in terms of a random dynamical system (that is, a skew product) with expanding fiber maps $\Psi_{R_i}$. An orbit of such a random dynamical system has the form
	\begin{equation}\label{randomOrbit}
		\theta, \Psi_{R_1}(\theta), (\Psi_{R_2}\circ\Psi_{R_2})(\theta), (\Psi_{R_3}\circ\Psi_{R_2}\circ\Psi_{R_1})(\theta),\dots
	\end{equation}
	This orbit is the sequence of post-collision angles of the particle in this model. The random dynamical systems point of view allows us to consider a so-called averaged transfer operator $P: BV \rightarrow BV$ (analogue of the Markov operator for the involved Markov chain) on the space of functions of bounded variation $BV$. This space is convenient for proving
	the main goal in this paper, namely the
	statistical properties of the horizontal displacement of the particles in the tube.
	For this, we introduce the observable $X: (0,\pi) \rightarrow \R, X(\theta):= W/\tan(\theta)$, the (one-step) horizontal displacement. The main result of this paper is Theorem~\ref{thm:main}, which gives a Central Limit Theorem with nonstandard normalization of $X$. The nonstandard scaling comes from the fact that $X$ does not have a second moment, i.e., it is superdiffusive. Note that the expectation ${\E}_{\mu}(X) = 0$.
	
	\begin{thm}\label{thm:main}
		Write $S_n:=\sum_{i=1}^{n-1}X_i$ where $X_i:= X\circ (\Psi_{R_i} \circ \dots \circ \Psi_{R_0})$. Then $\frac{S_nX}{\sqrt{n\log n}}$ converges in distribution to a Gaussian random variable with mean zero and variance $W^2$ under the measure $\nu^{\otimes \Z}\otimes \mu$.
	\end{thm}
	
	A direct consequence of Theorem~\ref{thm:main} is that the horizontal position of a particle in the long run will not be influenced by the specific geometric parameters of the microstructures, such as the curvature, but only by the width of the tube $W$. A technical reason for this is that the variance of the Gaussian only depends on the tail $\mu(|X|>t)$, see Sections~\ref{subsec:ev} and~\ref{subsec:ev2}.
	Here the tancency of the microstructures to the tube boundary from assumption (M2) in Section~\ref{sec:micro} is crucial.
	Since the invariant measure $\mu$ does not depend on the specific shape of the microstructure, other than by the curvature being positive\footnote{If the curvature of all the curves that make up the boundary of a microstructure is $0$, some different, atomic, invariant measures are possible. See for example~\cite{FE}.}, and since $X$ does not depend on the microstructures, there will be no geometric parameters of the microstructure in the expression of this tail, see
	Lemma~\ref{lem:tailX}.
	
	There is still some effect on of the specific parameters of the microstructures on the particle movement, but this is only visible in the sequence of post-collision angles \eqref{randomOrbit}. As a consequence of the Lasota-Yorke inequality Proposition~\ref{prop:ly}, this sequence is mixing exponentially on average, that is
	$$\left|\int_{\left[0,1\right]^n}\mu(\Psi_{R_n}\circ \dots \circ \Psi_{R_1}(A) \cap B) - \mu(A)\mu(B) d(R_1, \dots, R_n) \right| \leq C\alpha^{n},$$
	where
	$\alpha \leq \frac{1}{3} + \frac{C}{\kappa_{min}},$
	and where $\kappa_{min}$ is the minimal curvature on the boundary of a microstructure. Thus, the larger the minimal curvature, the stronger the decay of correlations for the angle process.
	We note that small changes of the shape of a microstructures such as introducing points with zero curvature can change the mixing rate (for example make it polynomial), see for instance~\cite{CZ} and~\cite{Z} where this is considered for scatterers in deterministic billiards. The type of limit law, however, is expected to remain the same, see~\cite{NZ18}.

	{\bf Main ingredients of the proofs.}
	The proof of Theorem~\ref{thm:main} relies on  the use of the  Nagaev-Guivarc’h spectral method
	to the type of randomness considered here, as previously described.
	We consider the transfer operator of the fiber maps $\Psi_{R_i}$ (see formula~\eqref{eq:PR}) and using the measure $\nu$ according to which the random variables $R_i$ are distributed  (see again Section~\ref{sec:rtr}), we define the average transfer operator $P$ of the random dynamical system (see formula~\eqref{avTran}). For the transfer operator
	$P$, along with its perturbation $P_tf= P(e^{itX}f)$, we obtain a spectral gap in $BV$. We also obtain a required good continuity estimate (in $t$) for $\|P_t-P\|_{BV}$.
	For both ingredients, spectral gap and continuity, we will heavily use  the type of randomization described above, by which the random dynamical system is defined.\\

	{\bf Structure of the paper.} In Section~\ref{sec:basicDef} we give a precise description of the model and record
	some needed technical results. In Section~\ref{sec:rtr} we discuss the randomization, and the form of the
	transfer operator that comes with it. Section~\ref{sec:variation} contains the main variation estimates of the map
	$\Psi_{R_i}$. In Section~\ref{sec:continuity} we give the continuity estimates for the transfer operator perturbed with
	the displacement function $X$, as needed for the Nagaev method. The spectral decomposition (in
	$BV$) of the averaged operator, along with Lasota-Yorke inequalities, is obtained in Section~\ref{sec:LY}. The Lasota-Yorke inequality is done after the continuity estimate since the former will have similar but easier estimates than the latter. In Section~\ref{sec:CLT} we complete the proof of Theorem~\ref{thm:main}.
	
	\medskip
	
	{\bf Acknowledgement.} We are grateful for the diligence of the referees that helped us to improve clarity and accuracy of our paper.
	
	\section{Definition of model and basic calculations} \label{sec:basicDef}
	
	\subsection{The tube}\label{sec:tube}
	The {\em tube} is a bi-infinite strip $\R \times [0,W]$, bounded by two horizontal lines
	$\ell_0 = \R \times \{ 0 \}$ and $\ell_W = \R \times \{ W \}$ for some large $W > 0$.
	A particle moves with constant speed back and forth between $\ell_0$ and $\ell_W$.
	When the particle reaches $\ell_0$ or $\ell_W$ for the $i$-th time,
	it enters a microstructure $\mic_i$. It will bounce a {\bf bounded} number of times inside $\mic_i$ before exiting and moving to the other side of the tube gaining a horizontal displacement $X_i = W/\tan\theta_{i}^{\text{out}}$, see Figure~\ref{fig:tube}.

	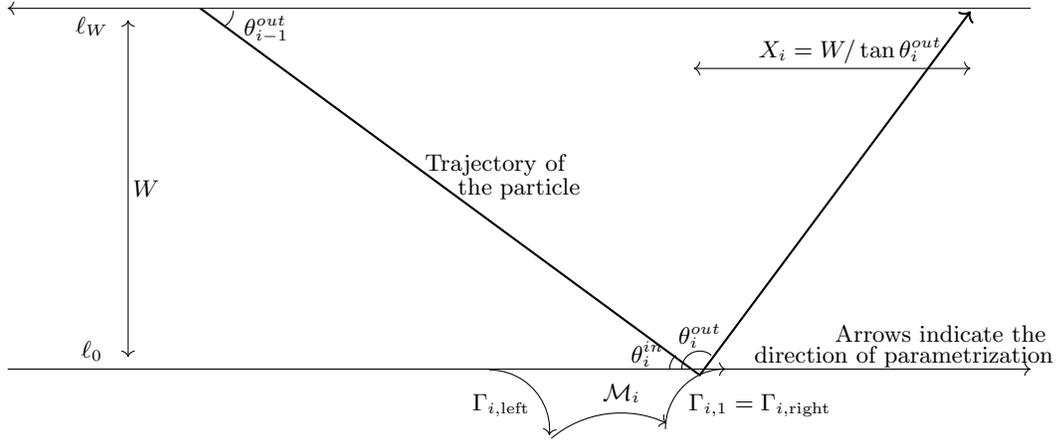
\begin{figure}[ht]
		\begin{center}
			\begin{tikzpicture}[scale=0.8]
				\draw[-] (-0.4, 6.7) arc (300:360:0.3); \node at (0.3, 6.65) {\small $\thetaout_{i-1}$};
				\draw[-] (7, 1) arc (180:130:0.3); \node at (6.63, 1.25) {\small $\thetain_i$};
				\draw[-] (7.2, 1) arc (180:50:0.3); \node at (7.5, 1.55) {\small $\thetaout_i$};
				\draw[->] (-4,1) -- (13,1); \node at (-2.6, 1.3) {\small $\ell_0$};
				\node at (-2.6, 6.7) {\small $\ell_{W}$};
				\draw[<-] (-4,7) -- (13,7);
				\draw[<->] (-2, 1.2) -- (-2,6.8); \node at (-1.7, 4) {\small $W$};
				\draw[<->] (7.4, 6) -- (12,6); \node at (10, 6.3) {\small $X_i = W/\tan \thetaout_i$};
				\node at (6.2, 0.6)  {\small $\mic_i$};
				\draw[->, thick] (-0.8,7) --  (7.5, 0.9) -- (12,6.95);
				\draw[<-] (5, -0.1) arc (-5:90:1); \node at (4.2, 0.4) {\small $\Gamma_i^{\text{left}}$};
				\draw[<-] (7.93, 1) arc (90:180:1); \node at (8.5, 0.4) {\small $\Gamma_{i,1} = \Gamma_i^{\text{right}}$};
				\draw[<-] (6.95, 0.1) arc (65:130:1.8);
				\node at (11.5, 1.6) {\small Arrows indicate the};
				\node at (10.9, 1.2) {\small direction of parametrization};
				\node at (4.1, 4.4) {\small Trajectory of};
				\node at (4.5, 4) {\small the particle};
			\end{tikzpicture}
			\caption{The tube with a piece of trajectory and microstructure $\mic_i$.}\label{fig:tube}
		\end{center}
	\end{figure}
	
	\begin{rmk}\label{rmk:angles}
		All angles will be between $\ell_0$ or $\ell_W$ (for $\thetaout_i$)  or tangent lines at collision points (for $\theta_{i,j}$)  and the outgoing
		trajectory.
		We use the direction of the parametrization, as in the Chernov \& Markarian book \cite{CM}, except that they use angles with the normal vectors, rather than with tangent lines, see Figure~\ref{fig:conv}.
		This convention means that $\ell_0$ is parametrized left-to-right and $\ell_W$ is parametrized right-to-left, as shown already in Figure~\ref{fig:tube}.
	\end{rmk}

	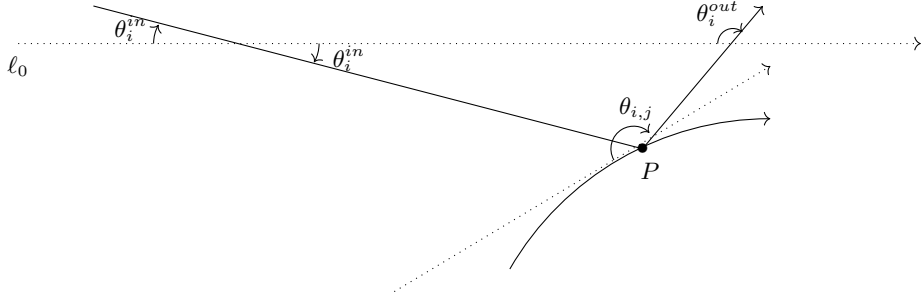
\begin{figure}[ht]
		\begin{center}
			\begin{tikzpicture}[scale=1]
				\draw[->] (-1,5.5) -- (6.3, 3.6) -- (7.9,5.5);
				\draw[<-] (8, 4) arc (90:150:4);
				\draw[->, dotted] (-2,5) -- (10,5);
				\draw[->, dotted] (3,1.7) -- (8,4.7);
				\node at (-2, 4.7) {\small $\ell_0$};
				\draw[->] (-0.2, 5) arc (180:150:0.5); \node at (-0.5, 5.2) {\small $\thetain_i$};
				\draw[->] (2, 5) arc (360:330:0.5); \node at (2.4, 4.8) {\small $\thetain_i$};
				\draw[->] (7.3, 5) arc (180:50:0.2); \node at (7.3, 5.4) {\small $\thetaout_i$};
				\draw[->] (5.92, 3.45) arc (210:45:0.3); \node at (6.22, 4.13)  {\small $\theta_{i,j}$};
				\node at (6.3, 3.6) {\small $\bullet$};   \node at (6.4,3.3) {\small $P$};
			\end{tikzpicture}
			\caption{Convention for the angles $\thetaout_i, \thetain_i$ and $\theta_{i,j}$.}\label{fig:conv}
		\end{center}
	\end{figure}

	Let us denote by $h_i$ the map that assigns 
	the $i$-th entrance coordinates $(\rin_i,\thetain_i)$ to the previous exit coordinates $(\rout_{i-1}, \thetaout_{i-1})$.
	The arrows of $\ell_0$ and $\ell_W$ as depicted in Figure~\ref{fig:tube} follow this convention for the tube, but for the open side of the microstructures $\mic_i$, the arrow should be reversed.
	The map $h_i$ therefore has the form
	\begin{equation}\label{eq:h}
		\begin{cases}
			h_i(\rout_{i-1}, \thetaout_{i-1})
			= (\rin_i, \thetain_i) = (\rout_{i-1} - W/\tan \thetaout_{i-1}, \thetaout_{i-1})\\[2mm]
			Dh_i = 
			\begin{pmatrix}
				1 & W/\sin^2 \thetaout_{i-1}, \\
				0 & 1
			\end{pmatrix}.	
		\end{cases}
	\end{equation}
	Thus $h_i$ represents the flight from $\ell_W$ to $\ell_0$ (or from $\ell_0$ back to $\ell_W$), and this flight has length $\tau_i = W/\sin \thetaout_{i-1}$.
	The term $W/\sin^2 \thetaout_{i-1}$ in $Dh_i$ in \eqref{eq:h} expresses the effect of a change in $\thetaout_{i-1}$
	on the horizontal displacement
	at the entrance on the other side of the tube; it depends on $W$ and even if it has no effect on $\thetain_i$, it will have an important effect on the next collision points and angles.

	\subsection{The microstructures}\label{sec:micro}
	
	Here we describe the shape of the microstructures and some properties of the trajectory of a particle within a microstructure. In Section~\ref{randomization} we explain how the position of a microstructure on the boundary of the tube is determined, given a random value $R_i$. This will also determine how a particle enters a microstructure. The microstructure $\mic_i$ entered at the $i$-th visit to the boundary of the tube is an area of unit length, with a boundary $\partial \mic_i$ made up of a finite number of smooth convex curves
	$$\partial \mic_i := \Gamma_{i}^{0} \cup \Gamma_i^{\text{left}} \cup \Gamma_i^1 \cup \dots \cup \Gamma_i^{k} \cup \Gamma_i^{\text{right}},$$
	for some fixed $k$. Here $\Gamma_{i}^{0}$ is the open side of $\mic_i$, so $\Gamma_{i}^{0}$ belongs to $\ell_0$ and $\ell_W$ alternately. We always set $\left|\Gamma_{i}^{0}\right| = 1$. Furthermore, $\Gamma_i^{\text{left}}$ and $\Gamma_i^{\text{right}}$ are the left and right curve of the microstructure directly adjacent to the tube boundary, called the \textit{left cheek} and \textit{right cheek} respectively. When we consider a trajectory inside $\mic_i$, we will only be interested in the boundary curves with which the particle has a collision. Given a trajectory with $n_i=n(\rin_i,\thetain_i)$ collisions with $\partial \mic_i$ during the $i$-th visit to a microstructure, we will denote the boundary curve at the $j$-th collision in $\mic_i$ by $\Gamma_{i,j}$, so with a subscript instead of a superscript, and where $j = 0, \dots, n_i$. Thus, during the trajectory in $\mic_i$ the particle starts at the open side $\Gamma_{i,0} = \Gamma_i^{0}$, then particle collides with $\Gamma_{i,1}$, next with $\Gamma_{i,2}$, etc.
	The lengths $|\Gamma_{i,0}|=1$.
	We use coordinates $(\rin_i, \thetain_i) \in \R \times (0,\pi)$,
	$i \in \Z$, for the entrance of the trajectory at the open side $\Gamma_{i,0}$ of $\mic_i$,
	and $(\rout_i, \thetaout_i)$ for the exit coordinates at $\Gamma_{i,0}$.
	
	In accordance with Remark~\ref{rmk:angles}, the boundary pieces are parametrized by $r$  oriented in such a way that $\mic_i$ is always to the left of the positively oriented tangent vector.
	
	Let $(r_{i,j}, \theta_{i,j})$ indicate the position and outgoing angle at the $j$-th collision with $\partial \mic_i$.
	The angle $\theta_{i,j}$ is measured with respect to the tangent line at the collision point $\Gamma_{i,j}(r_{i,j})$ following the convention of Remark~\ref{rmk:angles}, so
	$\theta_{i,j} \in [0,\pi]$, with grazing collisions for $\theta_{i,j} = 0$ or $\pi$.
	Note that
	\begin{equation}\label{eq:angles}
		(\rin_i,\thetain_i) = (r_{i,0}, \theta_{i,0}) \quad \text{ and } \quad 
		(\rout_i,\thetaout_i) = (r_{i,n_i}, \theta_{i,n_i}).
	\end{equation}
	The curvatures of $\partial \mic_i$ at the collision points $\Gamma_{i,j}(r_{i,j})$ are denoted as $\kappa_{i,j}$. The ``open'' side  $\Gamma_{i,0}$ of $\partial \mic_i$ is a straight arc, so $\kappa_{i,0} = \kappa_{i,n_i} = 0$.
	
	We assume that all microstructures have the same shape and size, and satisfy the following general conditions.
	
	\begin{enumerate}
		\item[(M1)] All $\Gamma_i^j$ are convex and there exist $0 < \kappa_{\min} \leq \kappa_{\max} < \infty$ such that the curvature
		$r \mapsto \kappa_i^j(r)$ are piecewise monotone functions bounded between $\kappa_{\min}$ and $\kappa_{\max}$. 
		Here $r$ is the parameter parametrizing $\Gamma_i^j$.
		\item[(M2)] The curves $\Gamma_{i}^{\text{left}}$ and $\Gamma_{i,}^{\text{right}}$ that are adjacent to
		the open side $\Gamma_{i,0}$ of $\mic_i$ are circle segments {\bf tangent} to the boundary of the tube, that is, at the intersection points, the tangent lines of $\Gamma_{i}^{\text{left}}$ and $\Gamma_{i,}^{\text{right}}$ are horizontal.
		\item[(M3)] There is $\gamma_0 > 0$ such that the angle $\gamma$ between any neighbouring pair of curves (other than $\Gamma_{i,0}$) lies in $(\gamma_0, \pi - \gamma_0)$.
		\item[(M4)] There is $\alpha_0 < \frac{\pi}{2}$ such that the normal vectors on $\Gamma_i$ pointing toward the tube have an angle $\alpha \in [-\alpha_0,\alpha_0]$ with the vertical direction pointing towards the tube.
	\end{enumerate}
	Assumption (M1) makes the fiber maps $\Psi_{R_i}$ expanding, which will be important for the proofs in the later sections.
	The tangency from assumption (M2) is crucial for obtaining the tail estimates $\mu(|X|> t)$ that underlie the superdiffusion, i.e., the need of nonstandard scaling in the Central Limit Theorem. Furthermore, this assumption gives a more predictable bouncing of particles hitting the boundary of the tube with small angles, not unlike the skipping stones over the surface of the water.
	Some parts of these assumption (i.e., that the curvature is uniformly bounded and bounded away from $0$ and that $\Gamma_i^{\text{left}}$
	and $\Gamma_i^{\text{right}}$ are circle segments) are somewhat artificial and probably needlessly strong, but they
	much reduce the technicalities of the proofs, as needed in for example Section~\ref{sec:variation}.
	
	Assumptions (M3) and (M4) imply that there is a uniform upper bound $N$ of the number of collisions $n_i$ that a particle can have before exiting $\mic_i$, see Lemma~\ref{lem:finite} below.
	Assumption (M3) prevents that $\partial \mic_i$ has cusps (essential for establishing the bound $N$), and it also prevents that there are two grazing collisions
	(i.e., $\sin \theta_{i,j-1} = \sin \theta_{i,j} = 0$) with arbitrarily small intermediate flight time. A somewhat more technical version of this fact is stated in the next lemma; it is used this way later in Lemma~\ref{lem:VarPsi}. 
	
	\begin{lemma}\label{lem:triangle}
		Under Assumption (M3) above, there is
		$K = K(\gamma_0) > 0$ such that
		such that
		\begin{equation}\label{eq:eta}
			\inf_{i,j} \tau_{i,j} \kappa_{i,j}\kappa_{i,j-1}
			+ \kappa_{i,j-1} \sin \theta_{i,j} + \kappa_{i,j} \sin \theta_{i,j-1} \geq K.
		\end{equation}
	\end{lemma}

	\begin{figure}[ht]
		\begin{center}
			\begin{tikzpicture}[scale=2]
				\draw[-] (5,0) arc (30:80:3); \node at (3,1.2) {\small $\Gamma_{i,j-1}$};
				\draw[-] (5,0) arc (150:100:3); \node at (7,1.2) {\small $\Gamma_{i,j}$};
				\draw[-] (4.9, 0.17) arc (120:60:0.2); \node at (5,0.3) {\small $\gamma$};
				\draw[-] (4.82, 0.6) arc (125:55:0.2); \node at (5,0.75) {\small $\gamma'$};
				\node at (4.05,1) {\small $\bullet$};  \node at (4,0.8) {\small $P_{j-1}$};
				\node at (6.15,1.15) {\small $\bullet$};  \node at (6.2,0.9) {\small $P_j$};
				\node at (5,1.2) {\small $\tau_{i,j}$};
				\draw[-] (3,1.6)--(5.8,0);
				\draw[-] (7,1.6)--(4,0);
				\draw[->,thick] (6.15,1.15)--(4.95,0.5)--(4.05,1)--(6.15,1.15)--(6.7,1.8);
				\draw[-,thick] (3.6,1.7)--(4.05,1);
			\end{tikzpicture}
			\caption{The triangle with angles $\theta_{i,j-1}$, $\theta_{i,j}$ and $\gamma' > \gamma$, $\sin \gamma > \gamma_0$.}\label{fig:triangle}
		\end{center}
	\end{figure}
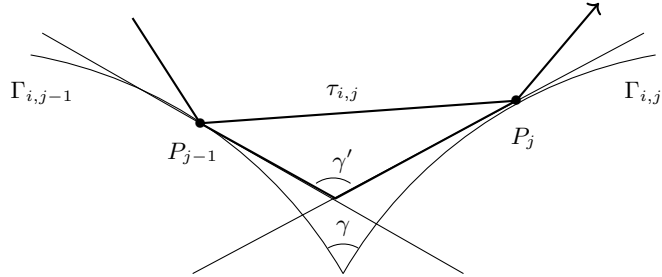

	\begin{proof}
		All curvatures $\kappa_{i,j}$ are bounded away from zero, so the three terms in \eqref{eq:eta} are non-negative. Thus it suffices to show that no more than two of these terms can be arbitrarily small.
		Assume that the $j$-th flight-time $\tau_{i,j}$ is very small. Then the $j-1$-st and $j$-th collision points $P_{j-1}$ ad $P_j$ are on neighbouring arcs of $\Gamma_i$.
		Consider the triangle as in Figure~\ref{fig:triangle}; its angles are $\pi-\theta_{i,j-1}$, $\pi-\theta_{i,j}$ and $\gamma'$ where $\gamma' > \gamma$, which is the angle between $\Gamma_{i,j-1}$ and $\Gamma_{i,j}$.
		This angle $\gamma' \to \gamma$ as $\tau_{i,j} \to 0$, so
		we can assume that $\gamma' \leq (\pi+\gamma)/2$
		for small $\tau_{i,j}$.
		Then $\theta_{i,j-1} + \theta_{i,j} = \pi + \gamma' \leq (3\pi+\gamma)/2$
		is bounded away from $2\pi$.
		Thus for small $\tau_{i,j}$, the angles
		$\theta_{i,j-1} $ and $\theta_{i,j}$ cannot be simultaneously close to $\pi$. The lemma follows from this.
	\end{proof}

	\begin{lemma}\label{lem:finite}
		Under (M2) and (M4), the number of collisions during a visit to a microstructure is bounded.
	\end{lemma}

	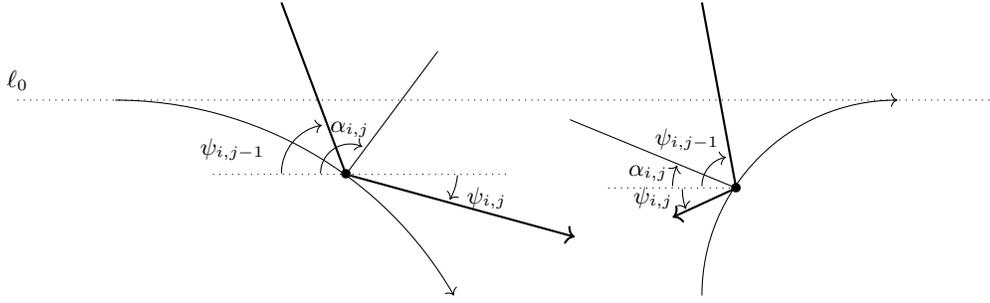
\begin{figure}[ht]
		\begin{center}
			\begin{tikzpicture}[scale=1.3]
				\draw[->, thick] (1.7,5)--(2.36,3.24)--(4.7, 2.6);
				\draw[->, thick] (6,5)--(6.35, 3.1)--(5.7, 2.8);
				\draw[<-] (8, 4) arc (90:180:2);
				\draw[-] (6.35, 3.1) -- (4.65,3.8);
				\draw[->] (0, 4) arc (90:30:4);
				\draw[-]  (2.36,3.24) -- (3.3,4.5);
				\draw[-, dotted] (-1,4) -- (9,4);
				\node at (-1, 4.2) {\small $\ell_0$};
				\draw[->] (2.1, 3.25) arc (180:65:0.3); \node at (2.4, 3.7) {\small $\alpha_{i,j}$};
				\draw[->] (1.7, 3.25) arc (180:100:0.5); \node at (1.2, 3.5) {\small $\psi_{i,j-1}$};
				\draw[->] (3.5, 3.23) arc (355:330:0.6); \node at (3.8, 3)  {\small $\psi_{i,j}$};
				\draw[->] (6, 3.12) arc (180:100:0.3); \node at (5.85, 3.6) {\small $\psi_{i,j-1}$};
				\draw[->] (5.7, 3.12) arc (180:155:0.5); \node at (5.45, 3.25) {\small $\alpha_{i,j}$};
				\draw[->] (5.8, 3.09) arc (180:200:0.6); \node at (5.5, 3)  {\small $\psi_{i,j}$};
				\node at (6.35, 3.1) {\small $\bullet$};
				\draw[-, dotted](6.35, 3.1)--(5,3.1);
				\node at (2.36,3.24) {\small $\bullet$};
				\draw[-, dotted](1, 3.24)--(4,3.24);
			\end{tikzpicture}
			\caption{Rules 1.\ and 2.\ in the proof of Lemma~\ref{lem:finite}.}\label{fig:finite}
		\end{center}
	\end{figure}

	\begin{proof}
		Condition (M2) excludes cusps, so there are $\eps>0$ and $\delta>0$ such that if $\tau_{i,j} < \eps$, then the angle between
		the tangent lines at collision point $r_{i,j-1}$ and $r_{i,j}$ is at least $\delta$.
		This prevents a trajectory from having many consecutive collisions near a single corner point.
		
		Let $\psi_{i,j}$ be the angle of the outgoing trajectory from collision point $r_{i,j}$ with the horizontal.
		We will assume that $\cos \psi_{i,j} \geq 0$; otherwise we can look at the image under left-to-right reflection and get the same result.
		Also assume that $\sin \psi_{i,j} > 0$, so the trajectory points away from the tube.
		
		Let $\alpha_{i,j}$ be the angle of the normal vector at
		$r_{i,j}$ with the horizontal.
		By assumption, $\alpha_{i,j} \in (0,\pi)$.
		Recalling (M4), we have two cases, see Figure~\ref{fig:finite}:
		
		\begin{enumerate}
			\item If $\alpha_{i,j} \in [\frac{\pi}{2},\frac{\pi}{2}+ \alpha_0]$,
			then $\psi_{i,j-1} \geq \alpha_{i,j} - \frac{\pi}{2}$. Therefore
			$$
			\psi_{i,j} = \psi_{i,j-1} + 2(\alpha_{i,j}-\psi_{i,j-1}) - \pi \leq \psi_{i,j-1} + 2\frac{\pi}{2} - \pi = \psi_{i,j-1}.
			$$
			This can happen without $\psi_{i,j-1}$ actually decreasing, namely at a grazing collision when $\psi_{i,j-1} = \alpha_{i,j} - \frac{\pi}{2}$. But according to Lemma~\ref{lem:triangle}, consecutive (almost) grazing collisions can only occur with a definite distance in between, so after a bounded number of collisions, $\psi_{i,j-1}$
			becomes negative, or the other case occurs.
			\item If $\alpha_{i,j} \in [\frac{\pi}{2}-\alpha_0,\frac{\pi}{2})$,
			then
			$$
			\psi_{i,j} = - \left( \psi_{i,j-1} - 2(\psi_{i,j-1} - \alpha_{i,j}) \right)
			= \psi_{i,j-1} - 2\alpha_{i,j} \leq \psi_{i,j-1} - (\pi-2\alpha_0).
			$$
		\end{enumerate}
		This means that starting with $\psi_{i,0} \in (0,\pi)$,
		$\psi_{i,j}$ will decrease with $j$, until,
		after a bounded number of steps, $\psi_{i,j} < 0$ and
		the trajectory will move towards the tube again.
		From this point onward, we can use rules for the time-reversed trajectory.
		This leads to the claimed bounded number of collisions.
	\end{proof}

	\subsection{Collision maps and their derivatives}\label{sec:derivatives}
	
	We define a two-dimensional map
	\begin{equation}
		\Psi^i = (\Psi^i_r,\Psi^i_\theta):(\rout_{i-1}, \thetaout_{i-1}) \mapsto (\rout_i, \thetaout_i),
		\label{twoDimPsi}
	\end{equation}
	mapping the exit coordinates of the previous microstructure $\mic_{i-1}$
	to the exit coordinates of the current microstructure $\mic_i$.

	Let $F_{i,j}:(r_{i,j-1}, \theta_{i,j-1}) \mapsto
	(r_{i,j}, \theta_{i,j})$
	be the collision map in $\mic_i$ and
	$\tau_{i,j}$ the lengths of the flights involved.
	By (2.26) from the Chernov \& Markarian book \cite{CM}, we have
	\begin{equation}\label{eq:DF}
		DF_{i,j}
		= \frac{-1}{\sin \theta_{i,j}}
		\begin{pmatrix}
			\tau_{i,j} \kappa_{i,j-1} + \sin \theta_{i,j-1} & \tau_{i,j} \\
			\tau_{i,j}\kappa_{i,j}\kappa_{i,j-1} + \kappa_{i,j-1} \sin \theta_{i,j} + \kappa_{i,j} \sin \theta_{i,j-1} & \tau_{i,j} \kappa_{i,j} + \sin \theta_{i,j}
		\end{pmatrix},
	\end{equation}
	so apart from the initial minus sign, all the entries are positive. That is, $\theta_{i,j}$ and $r_{i,j}$ are decreasing functions, both of $r_{i,j-1}$ and of $\theta_{i,j-1}$.
	Composing these collision maps shows that the signs of the corresponding derivatives satisfy
	\begin{equation}\label{eq:sgn}
		\sgn \frac{d\theta_{i,j-1}}{d\thetaout_{i-1}} = 
		\sgn \frac{dr_{i,j-1}}{d\thetaout_{i-1}} = 
		-\sgn \frac{d\theta_{i,j}}{d\thetaout_{i-1}} = 
		- \sgn \frac{dr_{i,j}}{d\thetaout_{i-1}} \neq 0.
	\end{equation}
	The map $\Psi^i$ is obtained by composing the separate collision maps:
	$$
	\Psi^i = F_{i,n_i} \circ \cdots \circ F_{i,1} \circ h_i.
	$$
	
	\begin{rmk}
		The map $h$ is not a collision map, because $\ell_0$ and $\ell_W$ are not part of the boundary of the billiard table. However, when composed with
		a collision map, the composition is in the form \eqref{eq:DF}. Indeed,
		because $\kappa_{i-1,n_{i-1}} = \kappa_{i,0} = 0$,
		\begin{eqnarray}\label{eq:pr}
			DF_{i,1} \cdot Dh_i
			&=&
			\frac{-1}{\sin \theta_{i,1}}
			\begin{pmatrix}
				\sin \theta_{i,0} & \tau_{i,1} \\
				\kappa_{i,1} \sin \theta_{i,0} & \tau_{i,1} \kappa_{i,1} + \sin \theta_{i,1}
			\end{pmatrix}
			\cdot
			\begin{pmatrix}
				1 & \tau_i/\sin\theta_{i,0} \\
				0 & 1
			\end{pmatrix}  \nonumber \\
			&=&
			\frac{-1}{\sin \theta_{i,1}}
			\begin{pmatrix}
				\sin \theta_{i,0} & \tau_i + \tau_{i,1} \\
				\kappa_{i,1} \sin \theta_{i,0} & (\tau_i+\tau_{i,1}) \kappa_{i,1} + \sin \theta_{i,1}
			\end{pmatrix},
		\end{eqnarray}
		as is to be expected.
	\end{rmk}

	Now to get a lower bound for $|\frac{d}{d\thetaout_{i-1}} \Psi^i_\theta(\rout_{i-1},\thetaout_{i-1})|$,
	necessary to prove that randomized angle map $\Psi_{R_i}$ is expanding,
	we can look at the right bottom entry of the derivative matrix
	\begin{eqnarray}\label{eq:DPsi}
		D\Psi^i &=&  DF_{i,n_i} \cdots DF_{i,1} \cdot Dh_i \\
		&=& DF_{i,n_i} \cdots DF_{i,2} \cdot 
		\frac{-1}{\sin \theta_{i,1}}
		\begin{pmatrix}
			\sin \theta_{i,0} & \tau_i + \tau_{i,1} \\
			\kappa_{i,1} \sin \theta_{i,0} & (\tau_i+\tau_{i,1}) \kappa_{i,1} + \sin \theta_{i,1}
		\end{pmatrix}, \nonumber
	\end{eqnarray}
	where we used \eqref{eq:pr} to get the second line. Note that $\tau_i
	= W/\sin \thetain_i$.
	Just multiplying the right bottom entries of each matrix and ignoring the factors $-1$,
	we obtain a term
	\begin{eqnarray}\label{eq:W}
		\left(1+\frac{(\tau_i+\tau_{i,0}) \kappa_{i,1}}{\sin \theta_{i,1}}\right)
		\prod_{j=2}^{n_i} \left(1+\frac{\tau_{i,j} \kappa_{i,j}}{\sin \theta_{i,j}}\right)
		&\geq& 1+\frac{\tau_1 \kappa_1}{\max(\sin \thetain_i, \sin \thetaout_i)} \nonumber \\
		&\geq& 1+ \frac{W \kappa_{i,1}}{\sin \thetain_i \max(\sin \thetain_i, \sin \thetaout_i)}.
	\end{eqnarray}
	The other terms all have the same sign, so the total derivative
	$\frac{d}{d\thetaout_{i-1}} \Psi^i_\theta(\rout_{i-1},\thetaout_{i-1})$
	is only larger in absolutely value. The next lemma is necessary for the estimates of the variation of $1/|\Psi'_{R_i}|$ done in Section~\ref{sec:variation}.
	
	\begin{lemma}\label{lem:tau}
		The flight-times $\tau_{i,j}$ as functions of $\thetaout_{i-1}$
		have at most two monotone branches on each piece of continuity.
	\end{lemma}
	
	\begin{proof}
		Recall that $\Gamma_{i,j-1}$ and $\Gamma_{i,j}$ are the pieces of the boundary of the $i$-th microstructure that the particle has its $j-1$-st and $j$-th collision with,
		namely at the collision points $\Gamma_{i,j-1}(r_{i,j-1})$ and
		$\Gamma_{i,j}(r_{i,j})$. Since the initial position in a microstructure is fixed for a given $R_i$ and since we only vary $\thetaout_{i-1}$, $r_{i,j-1}$ and $r_{i,j}$ are functions of $\thetaout_{i-1}$ through iterations of $F_{i,j}$, and the same holds for the outgoing angles $\theta_{i,j-1}$ and $\theta_{i,j}$ at these collision points.
		Let $A_{i,j} := [\Gamma_{i,j-1}(r_{i,j-1}), \Gamma_{i,j}(r_{i,j})]$
		be the straight arc between these two consecutive collision points,
		so the flight time $\tau_{i,j} = |A_{i,j}|$ is also a function of $\thetaout_{i-1}$.
		
		Assume by contradiction that $\thetaout_{i-1} \mapsto \tau_{i,j}(\thetaout_{i-1})$ has at least three monotone branches on an interval where it is well-defined and continuous. Then there is a local maximum in the interior of this interval.
		Hence there is a pair of distinct angles $\thetaout_{\pm}$ such that
		\begin{equation}\label{eq:thetapm}
			\tau_{i,j}(\thetaout_-) = \tau_{i,j}(\thetaout_+) =: \hat \tau,
			\quad \text{ and } \quad 
			\tau_{i,j}(\thetaout_{i-1}) > \hat \tau \quad \text{ for all }
			\thetaout_{i-1} \in (\thetaout_-, \thetaout_+).
		\end{equation}
		Let $r_{i,j-1}^{\pm} = r_{i,j-1}(\thetaout_{\pm})$, $r_{i,j}^{\pm} = r_{i,j}(\thetaout_{\pm})$ and $A^{\pm} = A_{i,j}(\thetaout_{\pm})$.
		Also define the straight arcs
		$$
		A' = [ \Gamma_{i,j-1}(r_{i,j-1}^-), \Gamma_{i,j-1}(r_{i,j-1}^+)]
		\quad  \text{ and } \quad
		A'' = [ \Gamma_{i,j}(r_{i,j}^-), \Gamma_{i,j}(r_{i,j}^+)].
		$$
		Then $A^+$, $A^-$, $A'$ and $A''$ define a quadrilateral $Q$ with two equal sides, see Figure~\ref{fig:Quadri}.
		
		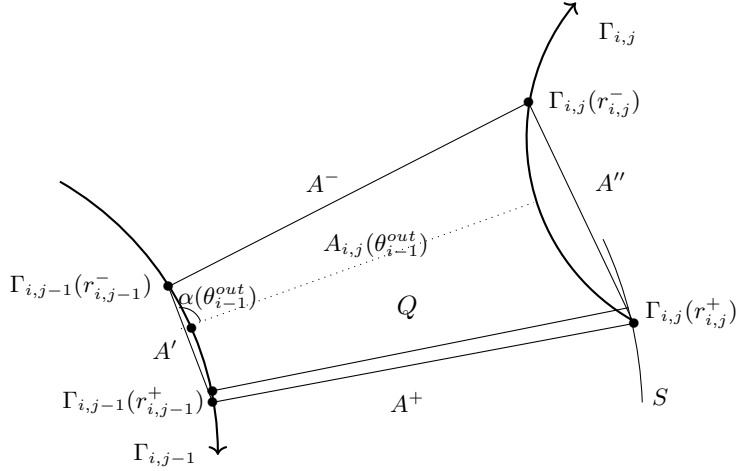
\begin{figure}[ht]
			\begin{center}
				\begin{tikzpicture}[scale=1.4]
					\draw[<-, thick] (8, 0) arc (0:60:3); \node at (7.5, 0) {\small $\Gamma_{i,j-1}$};
					\node at (7.53, 1.6) {\small $\bullet$};
					\node at (6.7, 1.6) {\small $\Gamma_{i,j-1}(r_{i,j-1}^-)$};
					\node at (7.95, 0.5) {\small $\bullet$};
					\node at (7.2, 0.5) {\small $\Gamma_{i,j-1}(r_{i,j-1}^+)$};
					\draw[<-, thick] (11.4, 4.3) arc (140:240:2); \node at (11.8, 4) {\small $\Gamma_{i,j}$};
					\node at (10.95, 3.35) {\small $\bullet$};
					\node at (11.95, 1.25) {\small $\bullet$};
					\draw[-] (7.53, 1.6) --  (10.95, 3.35) --  (11.95, 1.25) --  (7.95, 0.5)  -- (7.53, 1.6);  \node at (11.73, 2.6) {\small $A''$};
					\node at (7.5, 1) {\small $A'$};
					\node at (9, 2.6) {\small $A^-$};
					\node at (9.8, 0.5) {\small $A^+$};
					\node at (9.8, 1.4) {$Q$};
					\draw[-] (12.03, 0.5) arc (2:25:4);
					\node at (11.58, 3.35) {\small $\Gamma_{i,j}(r_{i,j}^-)$};
					\node at (12.5, 1.35) {\small $\Gamma_{i,j}(r_{i,j}^+)$};
					\node at (7.95, 0.6) {\small $\bullet$};
					\draw[-] (7.9, 0.6) --  (11.9, 1.4);
					\node at (7.75, 1.2) {\small $\bullet$};
					\draw[-, dotted] (7.65, 1.2) --  (11, 2.4);
					\draw[-] (7.63, 1.4) arc (95:20:0.2);
					\node at (8, 1.5) {\small $\alpha(\thetaout_{i-1})$};
					\node at (9.5,2) {\small $A_{i,j}(\thetaout_{i-1})$};
					\node at (12.2, 0.55) {\small $S$};    
				\end{tikzpicture}
				\caption{The quadrilateral $Q$ in the proof, with $|A^-| = |A^+| = \hat \tau$.}\label{fig:Quadri}
			\end{center}
		\end{figure}

		By convexity, the segment of $\Gamma_{i,j-1}$ between 
		$\Gamma_{i,j-1}(r_{i,j-1}^-)$ and $\Gamma_{i,j-1}(r_{i,j-1}^+)$
		and the segment of $\Gamma_{i,j}$ between 
		$\Gamma_{i,j}(r_{i,j}^-)$ and $\Gamma_{i,j}(r_{i,j}^+)$ lie inside $Q$.
		
		Orient $A'$ in the same way as $\Gamma_{i,j-1}$, and let $\alpha(\thetaout_{i-1})$ 
		be the angle between  $A'$ in its negative direction and $A_{i,j}(\thetaout_{i-1})$.
		Monotonicity of $\thetaout_{i-1} \mapsto \theta_{i,j-1}(\thetaout_{i-1})$ together with the convexity of
		$\Gamma_{i,j-1}$ imply that $\thetaout_{i-1} \mapsto \alpha(\thetaout_{i-1})$
		is monotone. Note that $r_{i,j} \mapsto \alpha(\thetaout_{i-1}(r_{i,j}))$
		is increasing, when we consider $\alpha$ as function of $r_{i,j}$,
		namely by taking the inverse function of $\thetaout_{i-1} \mapsto r_{i,j}(\thetaout_{i-1})$,
		which is decreasing due to \eqref{eq:sgn}.
		
		Let $\alpha^\pm = \alpha(\thetaout_\pm)$. The monotonicity of 
		$\thetaout_{i-1} \mapsto \alpha(\thetaout_{i-1})$ implies that $|A'| < |A''|$.
		As mentioned, the sides $A^+$and $A^-$ have equal length $\hat \tau$.
		
		Let $\beta^\pm$ be the internal angles of $Q$ where $A''$ meets with $A^\pm$.
		Since $|A''| > |A'|$, the smallest of $\beta^\pm$, say $\beta^+$,
		is a sharp angle.
		Therefore, the arc $A''$ near $\Gamma_{i,j}(r_{i,j}^+))$
		lies inside a circle of radius $\hat \tau$  with center
		$\Gamma_{i,j}(r_{i,j-1}^+)$,
		see the circle segment $S$ as in Figure~\ref{fig:Quadri}.
		
		Now take 
		$\alpha^\eps = \alpha(\thetaout_+-\eps) \in (\alpha^-,\alpha^+)$
		and 
		$r_{i,j-1}^\eps = r_{i,j-1}(\thetaout_+-\eps)$ for a small $\eps > 0$.
		Then it is impossible to fit a segment of length
		$\hat \tau$  between $\Gamma_{i,j}(r_{i,j}^+-\eps)$ and $\Gamma_{i,j}$
		at an angle $\alpha^\eps  \in (\alpha^-,\alpha^+)$ with $A'$.
		So $\tau_{i,j}(\thetaout_--\eps) < \hat \tau$, contradicting \eqref{eq:thetapm}.
		This proves the lemma.
	\end{proof}

	\section{Randomization and transfer operators}
	\label{sec:rtr}
	
	\subsection{Randomization}\label{randomization}
	
	As mentioned in the beginning of Section~\ref{sec:basicDef},
	the randomness concerns the position of microstructure $\mic_i$ relative to
	the position $\rin_i$ where the trajectory crosses the boundary of the tube, namely, the left endpoint of $\mic_i$ is randomized to
	\begin{equation}\label{eq:randomm}
		m_i = \rin_i - R_i,
	\end{equation}
	where $R_i$ are independent identically distributed random variables, distributed 
	according to some probability measure $\nu$ such that 
	the Radon-Nikod\'ym derivatives $\frac{d\nu(R_i)}{d\Leb}$ exist and are bounded. (One can think of a uniform distribution: $R_i \simeq U([0,1])$.)
	
	\begin{rmk}\label{rmk:random}
		This choice of randomization has the advantage that the trajectory and its derivatives still follow the rules of non-random
		elastic billiards. The randomization only affects where 
		the particle collides with
		the closed sides of microstructure $\mic_i$. 
		This randomization doesn't neglect the expansion and sensitivity of the past trajectory. In particular, it doesn't ignore the expansion built up due to the width of the tube, which the model considered in \cite{FZ10, FZ12} seems to ignore. 
		Ignoring the size of the tube would be physically
		inconsistent with the fact that the width of the tube $W$ features in the displacement $X_i = W/\tan \thetaout_i$ after exiting $\mic_i$. 
		For us, a large value of $W$ is crucial to get enough expansion for the Lasota-Yorke inequalities to hold.
	\end{rmk}
	The random version of $\Psi^i$ is a one-dimensional random map, denoted as:
	\begin{equation}\label{eq:psri}
		\Psi_{R_i}: (0,\pi) \to (0,\pi), \qquad  \thetaout_{i-1} \mapsto \thetaout_i,
	\end{equation}
	where $R_i$ is the sequence of random variables introduced in~\eqref{eq:randomm}. To be more precise, $\Psi_{R_i}(\thetaout_{i-1})$ is the angle component of the image of $\Psi^{i}(\rout_{i-1}, \thetaout_{i-1})$, with $\Psi$ defined in~\eqref{twoDimPsi}, where the random shift $R_{i}$ is taken into account after the particle is transported to the other side of the tube according to the function $h_i$. Notice that we can always set $\rout_{i-1} = 0$, since the periodic arrangement of microstructure is $\Z$-shift invariant (remember that the open side of each microstructure has length 1), and the precise position on the previous side of the tube modulo $\Z$ is forgotten after the random shift. Thus we obtain a random map of angles. Independent of the value of $R_i$, we can use the equation before \eqref{eq:W} to obtain
	\begin{equation} \label{twoDifs}
		|\Psi'_{R_i}(\Psi_{R_i'}^{-1}(\theta))| \geq  1+ \frac{W \kappa_{i,1}}{\sin \thetain_i \max\{ \sin \thetain_i, \sin \thetaout_i\}}.
	\end{equation}
	We can represent $\Psi_{R_i}$ as a skew-product with fiber map $\Psi_{R_0}$,
	as follows
	\begin{equation}\label{rds}
		T: [0,1]^{\Z}\times (0,\pi)\to [0,1]^{\Z}\times (0,\pi),\quad
		T ( (R_i)_{i \in \Z} ,\theta) = (\sigma((R_i)_{i \in \Z}), \Psi_{R_0}(\theta) ),
	\end{equation}
	where $\sigma$ is the usual left shift. The infinite product measure $\nu^{\otimes \Z}$ lives on $[0,1]^\Z$ and is left-shift invariant.
	The iterates of $T^n$ are given by
	$T^n ( (R_i)_{i \in \Z} ,\theta) = (\sigma^n((R_i)_{i \in \Z}), \Psi_{R_{n-1}} \circ \dots \circ \Psi_{R_0}(\theta))$.
	
	The whole random billiards seen as a $\Z$-extension over a compact billiard table is given by the skew-product with one extra component $u \in \Z$:
	\begin{equation}\label{eq:skew}
		( (R_i)_{i \in \Z} ,\theta, u) \mapsto (\sigma((R_i)_{i \in \Z}), \Psi_{R_0}(\theta), u + \xi(\theta)),
	\end{equation}
	where
	\begin{equation}\label{eq:xiDef}
		\xi(\theta) = \lfloor (\rout_i-m_i) + W/\tan\theta \rfloor  ).
	\end{equation}
	

	\subsection{Transfer operators (average and perturbed) and the BV space for one dimensional maps}\label{transOpSec}
	
	Before writing down the an explicit formula for the transfer operators, we have to describe $\Psi_{R_i}$ as an interval map. For every $R_i \in \left[0,1\right]$ the map $\Psi_{R_i}: (0,\pi) \rightarrow (0,\pi)$ has singularities caused by the particle having a grazing collision or hitting a corner of $\mic_i$. The interval $(0,\pi)$ is partitioned into sub-intervals by the singularities on which $\Psi_{R_i}$ is continuous, called \textit{\em pieces of continuity}. We can specify the pieces of continuity per microstructure. The discrete horizontal displacement between two microstructures $\mic_{i-1}$ to $\mic_i$ is the integer $\xi \in \Z$ defined in \eqref{eq:xiDef}, so we can define a subset $J_{\xi} \subset (0,\pi)$ such that $\thetaout_{i-1} \in J_{\xi}$ means $\xi(\thetaout_{i-1}) = \xi$. The randomisation will not affect this since the microstructure $\mic_{i}$ is shifted around the next entrance point $\rin_{i}$ independent of the value of $\xi$ by definition of the randomisation given in \eqref{eq:randomm}. Since the numbers of collisions within a single microstructure is uniformly bounded, see Lemma~\ref{lem:finite}, the number of pieces of continuity of $\Psi_{R_i}$ inside $J_{\xi}$ is finite (and in fact three if $\left|\xi\right|$ is large). Restricted to a piece of continuity, $\Psi_{R_i}$ need not be monotone. So we will subdivide the pieces of continuity into maximal subintervals (called \textit{domains of monotonicity}), where $\Psi_{R_i}$ is monotone as well.
	This will be important for computing the variation. The domains of monotonicity within $J_{\xi}$ will be denoted by $J_{\xi,\ell}$ and the index set by $\Lambda_{\xi} = \Lambda_{\xi,R_i}$. 
	
	Our transfer operators will be with respect to the invariant measure $d\mu = \frac12 \sin \theta d\Leb$ given by
	$\int P_{R_i} f \cdot g d\mu = \int f \cdot g \circ \Psi_{R_i} d\mu$. In the random setting of $\Psi_{R_i}$, $i \in \Z$, the pointwise formula for the transfer operator $P_{R_i}$ (defined w.r.t. $\mu$) takes the form
	\begin{equation}\label{eq:PR}
		P_{R_i}f(\theta) = \sum_{\xi \in \Z, \ell \in \Lambda_{\xi}}
		\frac{ f(\Psi_{R_i}^{-1}(\theta)) \sin(\Psi_{R_i}^{-1}(\theta))}
		{|\Psi'_{R_i}( \Psi_{R_i}^{-1}(\theta))| \sin \theta}
		\1_{J_{\xi,\ell}}(\Psi_{R_i}^{-1}\theta).
	\end{equation}
	The average transfer operator is given by
	\begin{equation}\label{avTran}
		Pf(\theta) = \int_0^1 \sum_{\xi\in \Z, \ell \in \Lambda_{\xi}}
		\frac{ f(\Psi_{R_i}^{-1}(\theta)) \sin(\Psi_{R_i'}^{-1}(\theta))}
		{|\Psi'_{R_i}(\Psi_{R_i}^{-1}(\theta))| \sin \theta}
		\1_{J_{\xi,\ell}}(\Psi_{R_i}^{-1}\theta) \, d\nu(R_i).
	\end{equation}
	
	\begin{rmk}\label{rmk:integr}
		In~\eqref{avTran} and throughout, 
		$\int_{[0,1]}\cdot\ d\nu(R_i)$ is shorthand for $\int_{[0,1]^\Z}\cdot\ d\nu^{\otimes \Z}\left((R_j)_{j\in\Z}\right)$. This is justified because the integrand depends on $R_i$ only and it is independent of $R_j$, $j\ne i$.
	\end{rmk}

	To obtain the desired limit theorem we consider a perturbed version of 
	$P$ by $e^{itX}$, $t\in\R$, where $X(\theta) = W/ \tan (\theta)$.
	The perturbed averaged operator is defined by
	\begin{equation}\label{eq:avpert}
		P_tf(\theta)=\int_0^1 P_{R_i} f(e^{itX})(\theta) \, d\nu(R_i).
	\end{equation}
	We note that $P_0=P$, as defined in~\eqref{avTran}.
	
	We want to apply the usual Nagaev method to $P_t$.
	In this sense, we need to establish Lasota-Yorke inequalities  and 'good' continuity estimates (as in Section~\ref{sec:continuity} below) in
	BV for $P_t$. Here and throughout,
	$$
	\|f\|_{BV}=\Var (f)+\|f\|_\infty, \quad
	\Var(f ) = \inf_{g\sim f} \sup_{0=y_0<\ldots<y_k=1}\sum_{j=1}^k | g(y_j) - g(y_{j-1}) |,
	$$
	where $g\sim f$ if $f$ and $g$ differ on a null set. Here and throughout, $\Var(f)$ denotes the variation of the (equivalence class) of $f$.

	We  record two inequalities that we shall use throughout without further comments.
	While it is clear that 
	$\|Pf(\theta)\|_\infty\le \int_0^1 \|P_{R_i} f(\theta)\|_\infty d\nu(R_i)$,
	we clarify that $\Var (Pf(\theta))\le \int_0^1 \Var(P_{R_i} f(\theta))\,d\nu(R_i)$. The latter can be justified as follows
	\begin{eqnarray*}
		\Var\left( \int_0^1 P_{R_i} f \,d\nu(R_i) \right) &=& \int_0^\pi \left|\frac{d}{d\theta}\int_0^1 P_{R_i} f(\theta)\,d\nu(R_i)\right|\, d\theta \leq \int_0^\pi \int_0^1 \left|\frac{d}{d\theta}P_{R_i} f(\theta)\right|\,d\nu(R_i)\, d\theta \\
		&=& \int_0^1\int_0^\pi \left|\frac{d}{d\theta}P_{R_i} f(\theta)\right|\, d\theta\,d\nu(R_i)=\int_0^1 \Var(P_{R_i} f) \,d\nu(R_i).
	\end{eqnarray*}
	
	\subsection{A classical estimate for near-grazing collisions}
	
	In the remainder of the paper we abbreviate systematically
	\begin{equation}\label{eq:thetatildetheta}
		\tilde \theta := \thetaout_{i-1} \quad \text{ and } \quad \theta := \thetaout_i = \Psi_{R_i}(\tilde\theta).
	\end{equation}
	We also introduce a threshold $\eta > 0$ such that if $\sin \tilde \theta < \eta$, then there are only three collision patterns possible in the microstructure $\mic_i$, namely (when the particle enters $\mic_i$ from the left) a single collision with the left cheek, a single collision with the right cheek, or a single collision with the left cheek followed by a collision with the right cheek. The latter we call a double collision.
	(If the particle enters $\mic_i$ from the right, then we have to swap the word ``left'' and ``right'', but there are still these three collision patterns.)
	It follows that $\Psi_{R_i}$ has six branches 
	on the region $\sin \tilde \theta < \eta$,
	three for $\tilde \theta$ close to $0$ and another three for $\tilde \theta$ close to $\pi$.
	
	The following lemma compares $\sin\tilde \theta$ with $\sin\theta$ in these cases. It comes basically from~\cite[Propositions 8 and 9]{SV07}, but we give a proof for transparency and completeness.
	
	\begin{lemma}\label{lem:SV} Given (M1) and (M2),
		there exists a constant $C_\kappa>0$ (depending only on $\kappa$)
		such that when $\sin\tilde\theta <\eta$,
		\begin{equation}\label{eq:theta}
			C_\kappa^{-1} \sin^2 \theta \leq \sin \tilde \theta 
			\leq C_{\kappa} \sqrt{\sin \theta}.
		\end{equation} 	
	\end{lemma}

	This means that $W(\sin \theta)^{-1/2} \ll  \xi(\Psi_{R_i}^{-1}(\theta)) \ll W(\sin \theta)^{-2}$, where $\xi$ is the skew-function from \eqref{eq:skew}.
	Also $\sin \tilde\theta < \eta$ implies that $\sin \theta < \sqrt{\eta C_\kappa}$.

	\begin{figure}[ht]
		\begin{center}
			\begin{tikzpicture}[scale=1]
				\draw[->] (1,5) -- (6.3, 3.6) -- (7.9,5.5);
				\draw[-] (8, 4) arc (90:150:4);
				\draw[-] (8,0) -- (5.65,5);
				\draw[-] (8,0) -- (8,4);
				\draw[-, dotted] (0,5) -- (10,5);
				\draw[-, dotted] (0,4) -- (10,4);
				\node at (0, 4.2) {\small $\ell_0$};
				\draw[-, dotted] (4,3.65) -- (8,3.65);
				\node at (8.6, 3.6) {\small $\frac{1}{\kappa}\cos \beta$};
				\draw[-, dotted] (4,2.5) -- (8,4.5);
				\node at (3, 2.5) {\small tangent line};
				\draw[->] (2, 5) arc (360:330:0.5); \node at (2.45, 4.8) {\small $\tilde \theta$};
				\draw[->] (4.5, 3.65) arc (180:145:0.5); \node at (4.27, 3.85) {\small $\tilde \theta$};
				\draw[->] (6.9, 3.6) arc (0:50:0.6); \node at (7.2, 4.12) {\small $\pi-\theta$};
				\draw[->] (7.2, 5) arc (180:50:0.2); \node at (7.4, 5.4) {\small $\theta$};
				\draw[->] (7.98, 0.7) arc (90:115:0.7); \node at (7.85, 0.8)  {\small $\beta$};
				\draw[->] (6.15, 3.9) arc (110:45:0.3); \node at (6.28, 4.1)  {\small $\alpha$};
				\node at (6.3, 3.6) {\small $\bullet$};   \node at (6.2,3.3) {\small $P$};
				\node at (5.5, 4.6) {\small $\overbrace{\qquad\quad \qquad}^{s}$};
			\end{tikzpicture}
			\caption{Comparing $\theta$ to $\tilde \theta$ for $\sin \tilde \theta < \eta$ at a right cheek collision.}\label{fig:SV}
		\end{center}
	\end{figure}
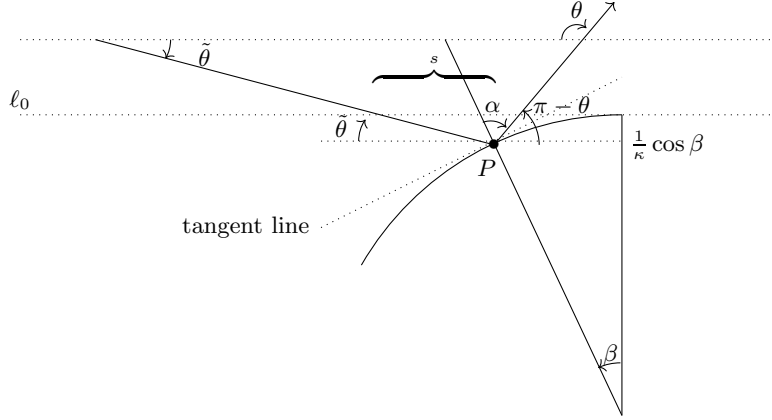

	\begin{proof}
		From Figure~\ref{fig:SV}, for the collision with the right cheek we find 
		$\alpha = \beta + \theta - \frac{\pi}{2}$,
		$(\pi - \theta) + \tilde \theta + 2\alpha = \pi$ and 
		$\frac{1}{\kappa}(1-\cos \beta) = s \sin \tilde\theta$
		for some $s \in (0,1)$.
		Rearrangement gives $\sin \beta = \sqrt{1-\cos^2 \beta}
		\sim \sqrt{2 \kappa s \sin \tilde\theta}$, 
		which combined with $\pi - \theta = \tilde \theta + \beta \leq 2\beta$
		gives $\sin \theta = \sin(\pi - \theta) \sim \sqrt{8 \kappa s \sin \tilde \theta} \leq \sqrt{8\kappa_{\max}}\sqrt{\sin \tilde \theta}$.
		Swapping the role of $\theta$ and $\tilde \theta$ (that is,
		looking at the collision with the left cheek), gives the other inequality. 
		
		For the double cheek collision, we use notation and estimates from Figure~\ref{fig:beta2},
		with $\theta = \thetaout_i$, $\tilde \theta = \thetaout_{i-1}$,
		$\theta_1 = \theta_{i,1}$ and $\theta_2 = \theta_{i,2}$.
		Also $\beta_1$ and $\beta_2$ are the angles at the collision points with $\Gamma_{i,1}$ and $\Gamma_{i,2}$ with curvatures $\kappa_1$ and $\kappa_2$, respectively.
		
		We have 
		\begin{equation}\label{eq:k2}
			\begin{cases}
				\beta_1+(\pi- \theta_1) = \tilde \theta,\quad \pi-\theta = \beta_2 + (\pi-\theta_2) \quad \text{ and } &
				\beta_2 - (\pi-\theta_2) = (\pi-\theta_1) - \beta_1 \\
				\frac{1}{\kappa_1}(1-\cos \beta_1) \sim
				\frac{1}{\kappa_2}(1-\cos \beta_2)
				+ s \sin(\pi-\theta_1-\beta_1)
				& \text{ for } s \lesssim 1.
			\end{cases}
		\end{equation}
		First assume that $\beta_2 - (\pi-\theta_2) = \pi-\theta_1 - \beta_1 \geq 0$, which means that the middle part of the trajectory in $\mic_i$ goes upwards, as in 
		Figure~\ref{fig:beta2}.
		This means that 
		$$
		2\beta_1 \leq \beta_1+(\pi- \theta_1) = \tilde \theta < 3\beta_1,
		$$
		where the second inequality follows because
		$s \sin(\pi-\theta_1-\beta_1) \leq \frac{1}{\kappa_1}(1-\cos \beta_1)$ to make a collision with $\Gamma_{i,2}$ possible.
		Using the main term in the Taylor expansions of $\cos$ and $\sin$, we can rewrite \eqref{eq:k2} to
		$$
		\frac{\beta_1^2}{2\kappa_1}
		\sim \frac{\beta_2^2}{2\kappa_2} + \beta_2 - (\pi-\theta_2).
		$$
		If $\beta_2 - (\pi - \theta_1) \gtrsim \frac{\beta_2^2}{2\kappa_2}$, then
		$$
		\sin \tilde \theta  \leq \sqrt{6\kappa \beta_2}
		\leq \sqrt{12\kappa_1 \, \sin \theta}.
		$$
		Otherwise, 
		$$
		\sin \tilde \theta \geq \sqrt{\frac{\kappa_1}{\kappa_2}}  \beta_2 
		\geq \sqrt{\frac{\kappa_1}{\kappa_2}} \sin \theta.
		$$
		If $\beta_2 - (\pi-\theta_2) = \pi-\theta_1 - \beta_1 \leq 0$,
		then we can reverse the roles of $\tilde \theta$ and $\theta$,
		and the analogous inequalities follow.
	\end{proof}
	
	%
	%

	\section{The variation of $1/|\Psi'_{R_i}|$ }\label{sec:variation}
	
	In this section we will consider the variation of $1/\left| \Psi'_{R_i}\right|$. We do this by considering the cases $\sin \tilde{\theta} \geq \eta$, when the absolute horizontal displacement $|\xi|$ is relatively small, and $\sin \tilde{\theta} \leq \eta$, when $|\xi|$ is large.
	
	\begin{lemma}\label{lem:Lambda}
		Assume (M3) and (M4).
		Let $\Lambda_\eta = \Lambda_\eta(W,R_i)$ refer to the set of branches of $\Psi_{R_i}: \tilde\theta \mapsto \theta$ on the subinterval of $(0,\pi)$ where $\sin \tilde\theta \geq \eta$.
		There is $C_\eta > 0$ such that $\# \Lambda_\eta \leq C_\eta W$.
	\end{lemma}
	
	\begin{proof}
		By the assumption on the microstructures, there is an upper bound $N$ on the number of branches associated to a single microstructure $\mic_i$.
		At angles $\tilde\theta$ satisfying $\sin \tilde\theta \geq \eta$, only microstructures with displacement $|\xi| \leq W/|\tan \tilde\theta| \leq W/\eta$ can be reached. So the lemma holds for 
		$C_\eta = N/\eta$.
	\end{proof}
	
	From Lemma~\ref{lem:tau} we can derive that
	there are a bounded number of pieces of monotonicity inside each piece of continuity $J$, and therefore $\Psi_{R_i}|_J$ has bounded variation.
	Let $I_\xi$ be the collection of branches associated to a displacement $\xi$, i.e.,
	$\xi(\tilde \theta) = \xi$ for each $\tilde \theta \in J_\xi$.
	The next lemma estimates the variation of $1/|\Psi'_{R_i}|$ restricted to each $J_\xi$  for $\sin\tilde \theta < \eta$
	(which means large values of $|\xi| \sim W/|\tan \tilde\theta|$),
	as function of $\theta = \Psi_{Ri}(\tilde \theta)$.
	On each monotone branch of $\Psi_{R_i}$, the variations in $\theta$ and in
	$\tilde\theta = \Psi_{R_i}^{-1}$ are the same.
	Also, the variation of a monotone function $\Var(f) \leq 2\|f\|_\infty$,
	or even $\Var(f) \leq \|f\|_\infty$ if $f$ is non-negative.
	If $f$ is non-negative with $N$ monotone branches, then $\Var(f) \leq N\|f\|_\infty$, which is why Lemma~\ref{lem:tau} is important in the next estimates.
	
	\begin{lemma}\label{lem:VarPsi}
		Assume properties (M1), (M3) and (M4) of the microstructures. Then
		$$
		\Var_{\theta}\left(\frac{1}{|\Psi_{R_i}'|}\Big|_{J_\xi}\right) = O(|\xi|^{-\frac32}) \quad \text{ as } |\xi| \to \infty.
		$$
	\end{lemma}
	
	\begin{proof}Throughout this proof we suppress the index $\tilde\theta$
		in the variation, and also write $R$ instead of $R_i$.
		Separating the entries with $\tau_i+\tau_{i,1}$ and $(\tau_i+\tau_{i,1})\kappa_{i,1}+\sin \theta_{i,1}$ in the rightmost matrix in \eqref{eq:DPsi}, obtain
		
		\begin{eqnarray}\label{eq:Psi}
			\Psi'_R|_{J_\xi} &=& \frac{\tau_i+\tau_{i,1}}{-\sin \theta_{i,1}} \frac{ \Omega(\sin \theta_{i,1}, \sin \theta_{i,2}, \dots, \sin \theta_{i,n_i-1},
				\tau_i,  \tau_{i,1}, \dots,  \tau_{i,n_i}, \kappa_{i,1}, \dots,  \kappa_{i,n_i}) }{\prod_{j=2}^{n_i} (-\sin \theta_{i,j}) }  \nonumber \\[2mm]
			&& + \ \frac{(\tau_i+\tau_{i,1})\kappa_{i,1}+\sin \theta_{i,1}}{-\sin \theta_{i,1}} 
			\frac{ \widehat\Omega(\sin \theta_{i,1}, \sin \theta_{i,2}, \dots, \sin \theta_{i,n_i-1},
				\tau_i,  \tau_{i,1}, \dots,  \tau_{i,n_i}, \kappa_{i,1}, \dots,  \kappa_{i,n_i}) }{\prod_{j=2}^{n_i} (-\sin \theta_{i,j}) }  \nonumber\\
			&=& \frac{(\tau_i+\tau_{i,1})\kappa_{i,1}+\sin \theta_{i,1}}{-\sin \theta_{i,1}} \times  \\
			&& \Big(
			\frac{\tau_i+\tau_{i,1}}{ (\tau_i+\tau_{i,1})\kappa_{i,1}+\sin \theta_{i,1} } \frac{ \Omega(\sin \theta_{i,1}, \sin \theta_{i,2}, \dots, \sin \theta_{i,n_i-1}, 
				\tau_i,  \tau_{i,1}, \dots,  \tau_{i,n_i}, \kappa_{i,1}, \dots,  \kappa_{i,n_i}) }{\prod_{j=2}^{n_i} (-\sin \theta_{i,j}) }  \nonumber \\[2mm]
			&& + \ 
			\frac{ \widehat\Omega(\sin \theta_{i,1}, \sin \theta_{i,2}, \dots, \sin \theta_{i,n_i-1},
				\tau_i,  \tau_{i,1}, \dots,  \tau_{i,n_i}, \kappa_{i,1}, \dots,  \kappa_{i,n_i}) }{\prod_{j=2}^{n_i} (-\sin \theta_{i,j}) } \Big), \nonumber
		\end{eqnarray}
		where $\Omega$ and $\widehat\Omega$ are  multivariate polynomials of their arguments.
		It follows that
		\begin{eqnarray}\label{eq:1Psi'}
			\frac{1}{\Psi'_R|_{J_\xi}} &=&
			\frac{-\sin \theta_{i,1}}{ (\tau_i+\tau_{i,1}) \kappa_{i,1}+\sin \theta_{i,1}} 
			\cdot \frac{ \prod_{j=2}^{n_i} (-\sin \theta_{i,j}) }{ \frac{ \tau_i+\tau_{i,1} }{ (\tau_i+\tau_{i,1})\kappa_{i,1}+\sin \theta_{i,1}}\, \Omega + \widehat\Omega}.
		\end{eqnarray}
		
		By our assumptions (see Section~\ref{sec:micro}), there are $n_i \leq N$ collisions. So, every  $\Psi_{R_i}|_{J_\xi}$ has a uniform bounded number of pieces of continuity.
		Below we show that the variation of $1/\Psi'_{R_i}$
		on each of these pieces is $O(1/W)$.
		
		Next, we use of the general formula
		\begin{eqnarray}\label{eq:varquot}
			\Var\left(\frac{g}{f}\right) &=& \sup_{x_0 < x_1 < \dots < x_r} \sum_{j=1}^r
			\left| \frac{g(x_j)}{f(x_j)} - \frac{g(x_{j-1})}{f(x_{j-1})}  \right|
			\nonumber \\
			&\leq&
			\sup_{x_0 < x_1 < \dots < x_r} \sum_{j=1}^r \frac{|g(x_j)| \ |f(x_j) - f(x_{j-1})| + |g(x_j)| \ |f(x_j) - f(x_{j-1})|}{|f(x_{j-1})f(x_j)|} \nonumber \\
			&\leq& \frac{\sup|g|\Var(f) + \sup|f| \Var(g)}{\inf|f|^2}.
		\end{eqnarray}
		Applying this for $f = (\tau_i + \tau_{i,1})\kappa_{i,1} + \sin \theta_{i,1} \geq \sqrt{W^2+\xi^2}$ and $g = \sin \theta_{i,1}$, we get
		\begin{eqnarray*}
			\Var\left( \frac{ -\sin \theta_{i,1} }{ (\tau_i+\tau_{i,1})\kappa_{i,1} + \sin \theta_{i,1} } \Big|_{J_\xi}\right)
			&\leq& \frac{1}{W^2 + \xi^2}
			\Big( \sup |\sin \theta_{i,1}| \Var((\tau_i+\tau_{i,1})\kappa_{i,1} + \sin \theta_{i,1} ) \\
			&& + \Var(\sin \theta_{i,1}) \sup((\tau_i+\tau_{i,1})\kappa_{i,1} + \sin \theta_{i,1}) \Big) \\
			&\ll& \frac{1}{W^2 + \xi^2} \left( \sqrt{\frac{W}{|\xi|}} + \sqrt{\frac{W}{|\xi|}}  \sqrt{W^2+\xi^2} \right) \qquad \text{ as } |\xi| \to \infty.
		\end{eqnarray*}
		This bound is summable over all displacements $\xi \in \Z$, and the best upper bound of the sum is independent of $W$.
		This makes $\Var\left(\frac{1}{|\Psi'|}\right) = O(1)$.

		It remains to show that the second factor in \eqref{eq:1Psi'}
		has bounded variation and supremum.
		
		The quotient 
		$\frac{ \tau_i+\tau_{i,1} }{ (\tau_i+\tau_{i,1})\kappa_{i,1}+\sin \theta_{i,1} }$
		is bounded by $\frac{1}{\kappa_{i,1}}$, bounded away from zero, and has at most four branches, so
		the variation is bounded by $4/\kappa_{\min}$.
		
		The factor
		$\Omega = \sum_{\ell=1}^{L_i} \Omega_\ell = \sum_{\ell=1}^{L_i} \prod_{k=2}^{n_i} \Omega_{\ell,k}$,
		where $\Omega_{\ell,k}$ is one of the four entries of the $2 \times 2$ matrix of $DF_k$ (without the prefactor $-1/\sin \theta_{i,j}$) in
		\eqref{eq:DF}, and $L_i$ is some bounded number, depending on the (bounded) number of collisions $n_i$.
		All these functions have bounded variation,
		so $\Var(\Omega) < \infty$, independently of $W$.
		The same holds for $\widehat\Omega$.
		Therefore the denominator
		$\frac{ \tau_i+\tau_{i,1} }{ (\tau_i+\tau_{i,1})\kappa_{i,1}+\sin \theta_{i,1}} \Omega + \hat{\Omega}$
		has bounded variation, and so has the numerator
		$\prod_{j=2}^{n_i} (-\sin \theta_{i,j})$.

		Next, we use \eqref{eq:varquot} for
		$f = \frac{ \tau_i+\tau_{i,1} }{ (\tau_i+\tau_{i,1})\kappa_{i,1}+\sin \theta_{i,1}} \Omega + \hat{\Omega}$ and $g \equiv 1$. For this,
		we need to show that $\inf|f|$ is bounded away from zero.
		This infimum $\inf|f|$ is  positive
		because $\hat{\Omega}$ has only positive terms, including
		$$
		\prod_{k=2}^{n_i} \left(\tau_{i,k} \kappa_{i,k}\kappa_{i,k-1} 
		+ \kappa_{i,k-1} \sin \theta_{i,k} 
		+ \kappa_{i,k} \sin \theta_{i,k-1}\right)$$ 
		obtained from taking the left bottom entries of the $DF_{i,k}$, $k=2, \dots, n_i$, in \eqref{eq:DF}. 
		According to Lemma~\ref{lem:triangle},
		this term is at least $K^{n-1} > 0$ (recall that $n_i \leq N < \infty$ by assumption).
		This ends the proof.
	\end{proof}
	
	An immediate consequence of Lemma~\ref{lem:VarPsi} is
	\begin{corollary}\label{cor:Q}
		There is a constant $C > 0$ such that for those $\xi \sim W/\tan \tilde\theta$ 
		corresponding to $\sin \tilde\theta < \eta$,
		$$ 
		\Var_{\theta}\left(\frac{\sin \tilde\theta}{|\Psi_{R_i}'(\tilde \theta)|}\Big|_{J_\xi}\right)  \leq C |\xi|^{-\frac52}.
		$$
	\end{corollary}
	
	\begin{proof}
		This follows from Lemma~\ref{lem:VarPsi} with the multiplication with the factor $\sin \tilde \theta$, 
		for which we notice that $\Var (\sin \tilde\theta|_{J_\xi})\le \sup_{\tilde\theta\in J_\xi}(\sin \tilde\theta)$.
	\end{proof}

	\section{Continuity estimates}\label{sec:continuity}
	
	In this section, we obtain the needed continuity estimate for the perturbed average operator $P_t$ defined in~\eqref{eq:avpert}.
	
	\begin{prop}\label{prop:cont} Let $f\in BV$. There exists $C_{BV}>0$ so that for all $\theta\in (0,\pi)$ and all $t\in\R$,
		$$
		\left\|(P_t-P_0)f\right\|_{BV}
		\le C_{BV}\,|t|\, \|f\|_{BV}
		$$
	\end{prop}
	The proof is carried out in the remainder of this section.
	
	\subsection{Continuity estimates using averaging for $\sin\theta<\eta$.}
	

	Due to (M1)-(M4), the microstructures are shaped so that a visiting trajectory has only a bounded number of collisions, so the map
	$\Psi_{R_i}:(0,\pi) \to (0,\pi)$ has finitely many branches {\bf associated to a single} microstructure. However, every microstructure can be reached by taking $\tilde \theta \in \Psi_{R_i}^{-1}(\theta)$ sufficiently close to $0$ or $\pi$. Therefore $\Psi_{R_i}$ has infinitely many branches, but the domains of these branches have only $0$ and $\pi$ as accumulation points.
	If $\sin \tilde\theta < \eta$ where $\eta$ is as in Lemma~\ref{lem:SV}, i.e., $|\xi(\tilde\theta)| \geq \xi_\eta \sim W/\eta$,
	and if the particle enters $\mic_i$ from the left, then there are only three branches associated to each microstructure, representing trajectories that
	\begin{enumerate}
		\item collide only with the left cheek of the microstructure: $\sin \theta < \sin \theta_{i,1} < \sin \tilde \theta$;
		\item collide only with the right cheek of the microstructure:
		$\sin \tilde \theta < \sin \theta_{i,1} < \sin \theta$;
		\item collide once with the left cheek and once with the right cheek of the microstructure:  $\sin \theta_{i,1} < \sin \tilde \theta$ and
		$\sin \theta_{i,2} < \sin\theta$.
	\end{enumerate}
	If the particle enters $\mic_i$ from the right, then
	the three above cases work with ``left'' and ``right'' swapped.
	
	The derivative of those branches $\geq W \kappa_{i,1}/\sin^2 \theta$
	according to \eqref{eq:W}.
	In cases 1. and 2. this bound is sharp, in case 3. there
	is another factor $\approx 1+\frac{s \kappa_{i,2}}{\sin \theta_{i,2}}$ associated to the second reflection
	in the microstructure and $s = \tau_{i,2} \approx 1$ (i.e., the width of the microstructure). Hence the derivative of the branch of case 3. is much larger.
	%
	
	\subsubsection{Estimating without using averaging}\label{sec:EstNoAv}
	Recall that $X(\theta) = W/\tan \theta$ and that
	\begin{equation}\label{eq:psm}
		P_{R_i} f \left(e^{itX}-1\right)(\theta) =
		\sum_{\xi \in \Z, \ell \in \Lambda_{\xi}}\frac{ f(\Psi_{R_i}^{-1}(\theta)) \sin(\Psi_{R_i}^{-1}(\theta))
			\Big( e^{ itX( \Psi_{R_i}^{-1}(\theta) )} - 1 \Big)}
		{|\Psi'_{R_i}( \Psi_{R_i}^{-1}(\tilde \theta))| \sin \theta}
		\1_{J_{\xi,\ell}}(\Psi_{R_i}^{-1}(\theta)).
	\end{equation}
	The continuity estimate of the transfer operator involves \eqref{eq:PR}
	with an extra factor $|e^{itX}-1| \geq |tX|/2$ for $X = W/\tan \Psi_{R_i}^{-1}(\theta)$. The estimate below suggests that without averaging,
	there is no hope to obtain the desired continuity estimate.
	For $\sin \theta \to 0$,
	using the only the ``left cheek'' where $\sin \tilde \theta \geq \sin \theta_{i,1} \geq \sin \theta$ in \eqref{eq:W}, we obtain, say for a positive $f$:
	\begin{align} \left|P_{R_i} f\left(e^{itX}-1\right) (\theta) \right|&
		\gg \sum_{\xi = W\max\{\frac{1}{\eta}, \frac{1}{\sqrt{ C_\kappa \sin \theta} } \} }^{\frac{W}{\sin \theta}}
		\frac{ f(\tilde \theta) \sin(\tilde \theta) \ |t| W/|\tan \tilde \theta|}
		{2\, |\Psi'_{R_i}(\tilde \theta)| \sin \theta}
		\1_{\xi(\tilde\theta) = \xi} \nonumber\\
		&\gg
		|t| W \sum_{\xi = W \max\{\frac{1}{\eta}, \frac{1}{ \sqrt{C_\kappa \sin \theta} } \}}^{\frac{W}{\sin \theta}}
		\frac{ f(\tilde \theta) |\cos(\tilde \theta)| }
		{2\, (W \kappa /\sin^2 \tilde \theta) \sin \theta}
		\1_{\xi(\tilde\theta) = \xi}  \nonumber  \\
		&\gg |t| W
		\sum_{\xi = W \max\{ \frac{1}{\eta}, \frac{1}{\sqrt{C_\kappa \sin \theta}} \}}^{\frac{W}{\sin \theta}}
		\frac{ \sin^2 \tilde \theta f(\tilde \theta)  }
		{2W \kappa \sin \theta}
		\1_{\xi(\tilde\theta) = \xi}\label{eq:Pest} \\
		&\gg \frac{|t| W}{2\kappa_{\max} \sin \theta} \sum_{\xi = W\max\{ \frac{1}{\eta}, \frac{1}{\sqrt{C_\kappa\sin \theta}} \}}^{\frac{W}{\sin \theta}}
		\frac{W}{\xi^2} f(\tilde \theta) \nonumber \gg \inf|f| \frac{|t| \sqrt{C_\kappa} W}{4 \kappa_{\max} \sqrt{\sin \theta}}, \nonumber
	\end{align}
	so this blows up as $\sin \theta \to 0$. This shows that the averaging in $Pf$ is crucial to obtain a useful continuity estimate.

	\subsubsection{Estimates using averaging}
	\label{sec:p}

	The main idea of exploiting the averaging for small values of $\sin \theta$
	is that 
	the integration over $d\nu(R_i)$ will be over a small subinterval $p(J_\xi,\tilde\theta)$
	of $[0,1]$, which leads to a gain of a small factor $\nu(p(J_\xi,\tilde\theta))$ inside the sum in~\eqref{eq:psm}.
	By our assumption on $\nu$, this is comparable to the length $|p(J_\xi,\tilde\theta)|$,
	and this multiplication will lead to bounded sums, as argued below.
	
	If the exit angle $\theta = \thetaout_i$ is fixed, and the displacement $|\xi| > \xi_\eta$, then the inverse map $\Psi_{R_i}^{-1}$ has only three branches. But not all entrance positions
	agree with these branches. Depending on whether we look at the left cheek branch $J_{\xi,L}$, double cheek branch $J_{\xi,D}$ or right cheek branch $J_{\xi,R}$, there is a different interval of possible entrance positions.
	This means that a different subinterval $p(J_{\xi, \cdot}, \tilde\theta) \subset [0,1]$ 
	of $R_i$-values such that the random shift of the position of $\mic_i$ realizes the required entrance.

	In the following illustrating computation, all summands are non-negative,
	and thus, we can swap an integral and an infinite sum. 
	Recall from Lemma~\ref{lem:SV} that $\xi_\eta  \sim W/\eta$ is a lower bound for the absolute value of all displacements when $\sin\tilde\theta < \eta$.
	
	\begin{eqnarray}\label{eq:avadv}
		\nonumber Pf(\theta) \Big|_{\{ \sin\tilde\theta < \eta\}} &=&
		\int_0^1 \sum_{\Lambda_\eta} \frac{f(\tilde \theta)}{|\Psi'_{R_i}(\tilde \theta)|} \frac{\sin \tilde \theta}{\sin \theta} \1_{J_{\xi,\ell}}(\tilde \theta)\, d\nu(R_i) \\
		&=&
		\int_0^1 \sum_{|\xi| \geq \xi_\eta}
		\sum_{\ell \in \{L,R,D\}} \frac{f(\tilde \theta)}{|\Psi'_{R_i}(\tilde \theta)|} \frac{\sin \tilde \theta}{\sin \theta} \1_{J_{\xi,\ell}}(\tilde \theta)\, d\nu(R_i) \nonumber\\
		&=&
		\sum_{|\xi| \geq \xi_\eta}
		\sum_{\ell \in \{L,R,D\}}  \int_0^1 \frac{f(\tilde \theta)}{|\Psi'_{R_i}(\tilde \theta)|} \frac{\sin \tilde \theta}{\sin \theta} \1_{J_{\xi,\ell}}(\tilde \theta) \, d\nu(R_i) \nonumber\\
		&=&
		\sum_{|\xi| \geq \xi_\eta}
		\sum_{\ell \in \{ L,R,D \} }  \int_{p(J_{\xi,\ell},\tilde\theta)} \frac{f(\tilde \theta)}{|\Psi'_{R_i}(\tilde \theta)|} \frac{\sin \tilde \theta}{\sin \theta} \1_{J_{\xi,\ell}}(\tilde \theta) \, d\nu(R_i).
	\end{eqnarray}
	
	The interval $p=p(J_\xi,\tilde\theta)$ is portrayed in Figure~\ref{fig:p}.

	\begin{figure}[ht]
		\begin{center}
			\begin{tikzpicture}[scale=0.9]
				\draw[->] (-0.3, 7) arc (360:330:0.5); \node at (-0.2, 7.3) {\small $\tilde \theta_+$};
				\draw[->] (0.1, 7) arc (360:330:1.4); \node at (0.3, 7.3) {\small $\tilde \theta_-$};
				\draw[-] (7.2, 1) arc (180:60:0.3); \node at (7.5, 1.55) {\small $\theta$};
				\draw[->] (-2,1) -- (13,1); \node at (-1.6, 1.2) {\small $\ell_0$};
				\draw[->] (-2,7) -- (13,7);
				\draw[<->] (11, 1.2) -- (11,6.8); \node at (11.3, 4) {\small $W$};
				\node at (-1.6, 7.2)  {\small $\ell_W$};
				\draw[-, thick] (7,1.03) -- (7.4, 1.03);
				\node at (7.2, 1.3)  {\small $p$};
				\draw[->] (-0.8,7) --  (7.5, 0.9) -- (10.6,5);
				\draw[->] (-1.2,7) --  (7.34, 0.8) -- (10.5,5);
				\draw[-] (5, 0) arc (0:90:1); \node at (4.5, 0.5) {\small $z_1$};
				\draw[-] (7.93, 1) arc (90:180:1); \node at (7.5, 0.5) {\small $z_3$};
				\draw[-] (7.05, -0.1) arc (30:150:1.2);
			\end{tikzpicture}
			\caption{The interval $p=p(J_\xi,\tilde\theta) \subset \Gamma_0$. We took $\sin \theta$ far from $0$ to make the picture clearer.}\label{fig:p}
		\end{center}
	\end{figure}
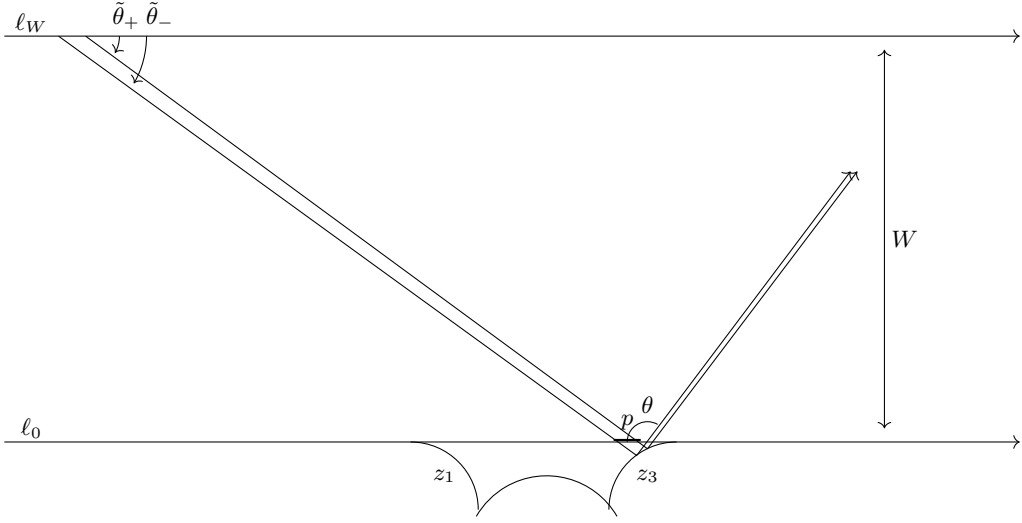

	In words,~\eqref{eq:avadv} tells us that estimating the integrand by its supremum (or its variation, if that is what we are interested in), we can replace the integral by a multiplication of $\nu(p(J_{\xi,j},\tilde\theta)) \leq h^+ \, |p(J_{\xi,j},\tilde\theta)|$ for $h^+ := \sup \frac{d\nu}{d\Leb}$, which gives the mentioned factor in the estimates of the sums.
	More precisely, using~\eqref{eq:avadv}, 
	\begin{align} \label{eq:psm2}
		\Big\|P_{R_i} f &  \left(e^{itX}-1\right) \Big|_{\{ \sin\tilde\theta < \eta\}}\Big\|_\infty \nonumber \\
		&\le \sup_{\theta}
		\sum_{\xi = W(C_\kappa^{-1} \sin \theta)^{-1/2}}^{W(C_\kappa\sin \theta)^{-2}}
		\left|\frac{ f(\tilde \theta) \sin(\tilde \theta) \ |t| W/|\tan \tilde \theta|}
		{| \Psi_{R_i}'(\tilde \theta)| \sin \theta}\right| \, \nu\left(p(J_\xi,\theta)\right)\,
		\1_{\xi(\tilde\theta) = \xi}.
	\end{align}
	
	A similar argument can be used to bound the variation, again starting from \eqref{eq:avadv}.
	The precise details are provided in Section~\ref{subsec:varsmall}.
	
	Recall that $p(J_{\xi,L/R},\tilde\theta)$ denotes the interval in $[0,1]$ obtained from looking at the left or right cheek branches, while $p(J_{\xi,D},\tilde\theta)$
	denotes the interval coming from looking at the double cheek branch.
	
	\begin{lemma}\label{lem:p}
		The following estimates for the left, right and double cheek collisions hold:
		$$
		\nu(p(J_{\xi,\ell},\tilde\theta)) \leq h^+ |p(J_{\xi,\ell},\tilde\theta)| \leq \frac{h^+ \, \max\{ |\tan\tilde \theta|^2 : \xi(\tilde\theta) = \xi\}}{2\kappa_{\min} W}
		\leq \frac{h^+ W}{2\kappa_{\min}}\frac{1}{ (|\xi|-1)^2},
		$$
		where $h^+ = \sup \frac{d\nu}{d\Leb}$ and  $\ell \in \{ L,R,D \}$.
	\end{lemma}
	
	\begin{figure}[ht]
		\begin{center}
			\begin{tikzpicture}[scale=1]
				\draw[->, thick] (1,5) -- (6.3, 3.6) -- (7.9,5.5);
				\draw[-] (8, 4) arc (90:150:4);
				\draw[-] (8,0) -- (5.65,5);
				\draw[-] (8,0) -- (8,4);
				\draw[-, dotted] (0,5) -- (10,5);
				\draw[-, dotted] (0,4) -- (10,4);
				\node at (0, 4.2) {\small $\ell_0$};
				\draw[->] (2, 5) arc (360:330:0.5); \node at (2.45, 4.8) {\small $\tilde \theta_+$};
				\draw[->] (7.3, 5) arc (180:53:0.2); \node at (7.4, 5.4) {\small $\theta$};
				\draw[->] (8, 0.7) arc (90:115:0.7); \node at (7.8, 0.8)  {\small $\beta$};
				\draw[-] (6.15, 3.9) arc (110:45:0.3); \node at (6.32, 4.1)  {\small $\alpha$};
				\node at (6.3, 3.6) {\small $\bullet$};   \node at (6.2,3.3) {\small $P$};
				\node at (1, 5) {\small $\bullet$};   \node at (1,5.3) {\small $A$};
				\node at (5.65, 5) {\small $\bullet$};   \node at (5.65,5.3) {\small $B$};
				\node at (7.5, 5) {\small $\bullet$};   \node at (7.5,4.7) {\small $C$};
			\end{tikzpicture}
			\caption{Relations between $\alpha,\beta,\theta$ and $\tilde \theta$.}\label{fig:beta}
		\end{center}
	\end{figure}
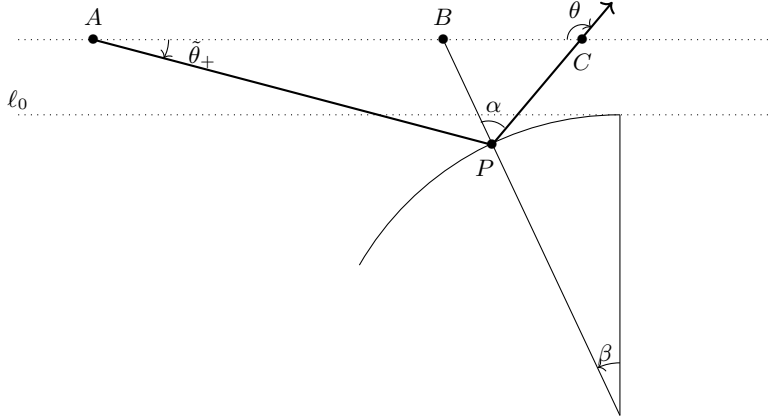

	\begin{proof}
		We start with the computation for the right cheek branch.
		Let $\tilde \theta_+$ and $\tilde \theta_-$ be the angles whose trajectories correspond to the end-points of $p$.
		The angles between the outgoing normal vectors with the vertical at the corresponding collision points at the left or right cheek of the microstructure are $\beta_+$ and $\beta_-$.
		Let $\alpha_+$ and $\alpha_-$ be the angles that the outgoing trajectories makes with the normal vectors at the two collision points.
		Adding up the angles in the triangles $APC$ and $BPC$ in
		Figure~\ref{fig:beta} (and doing the same for the trajectory with incoming angle $\tilde\theta_-$), we obtain
		$$
		\begin{cases}
			\theta = 2\alpha_+ + \tilde\theta_+, &\alpha_+ = \beta_+ + \theta - \frac{\pi}{2},\\
			\theta = 2\alpha_- + \tilde\theta_-,
			& \alpha_- = \beta_- + \theta - \frac{\pi}{2},
		\end{cases}
		\quad \text{ and } \qquad
		\left|\frac{W}{\tan \tilde \theta_+} -
		\frac{W}{\tan \tilde \theta_-}\right| \leq 1.
		$$
		So
		$$
		|\alpha_+-\alpha_-| =
		|\beta_+-\beta_-| = \frac12\, |\tilde \theta_+-\tilde\theta_-|
		\leq \frac{|\tan \tilde \theta^+ \tan \tilde \theta^-|}{2W}.
		$$
		Thus, by (M2),
		\begin{align*}
			|p(J_{\xi,R},\tilde\theta)| \leq \left|\frac{\sin \beta_+}{\kappa} -
			\frac{\sin \beta_-}{\kappa} \right| \ll \frac{\max\{ |\tan \tilde \theta|^2 : \xi(\tilde\theta) = \xi\} }{2\kappa W}.
		\end{align*}
		
		The estimate for the left cheek branch is the same.
		
		For the double cheek branch we have the following relations between the angles indicated in Figure~\ref{fig:beta2} (where we abbreviated $\theta_1 = \theta_{i,1}$ and $\theta_2 = \theta_{i,2})$:
		$$
		\tilde \theta = \beta_1+(\pi-\theta_1), \ \pi - \theta = \beta_2+\pi-\theta_2,\ \theta_1 - \beta_1 = \beta_2 - (\pi-\theta_2).
		$$
		This gives $\pi-\theta + \tilde \theta = 2(\beta_1+\beta_2) =: 2\beta$.
		As before, let $\tilde \theta^{\pm}$ be angle corresponding to
		the left-most and right-most entrance positions satisfying $\xi(\tilde\theta^+) = \xi(\tilde\theta^-) = \xi$,
		and let $\beta_1^\pm$ and $\beta_2^\pm$ and $\beta^\pm = \beta_1^\pm + \beta_2^\pm$ indicate the angle of the corresponding collision points.
		The collision points themselves satisfy $r_{i,1} \sim \beta_1/\kappa_1$ and
		$r_{i,2} \sim (\pi - \beta_2)/\kappa_2$ as arc-lengths of $\Gamma_{i,1}$ and $\Gamma_{i,2}$ with local curvatures $\kappa_1$ and $\kappa_2$,
		respectively. Therefore
		$$
		\sgn \frac{d \beta_1}{d\tilde\theta} = \sgn \frac{d r_{i,1}}{d\tilde\theta} = - \sgn \frac{dr_{i,2}}{d\tilde\theta}  =  \sgn \frac{d \beta_2}{d\tilde\theta}.
		$$
		This shows that $|\beta_1^+ - \beta_1^-| \leq |\beta^+ - \beta^-|$.
		
		As before $\left| \frac{W}{\tan \tilde \theta^+} - \frac{W}{\tan \tilde \theta^-} \right| \leq 1$.
		This gives, again due to (M2),
		$$
		|p(J_{\xi,D},\tilde\theta)| \leq \frac{|\sin \beta_1^+ - \sin \beta_1^-|}{\kappa_1} \leq \frac{|\beta_2-\beta_1|}{\kappa_1}
		\leq \frac{|\tilde\theta^+ - \tilde\theta^-|}{2\kappa_1}
		\leq \frac{\max\{ |\tan \tilde \theta|^2 : \xi(\tilde\theta) = \xi\}}{2\kappa_{\min} W}.
		$$
		Finally, we estimate the $\nu$-measure of the interval $p(J_{\xi,R/L/D}, \tilde\theta)$ by $h^+ \, |p(J_{\xi,R/L/D},\tilde\theta)|$, and recall that $|\tan \tilde\theta| \leq W/(|\xi(\tilde\theta)|-1)$.
	\end{proof}

	\begin{figure}[ht]
		\begin{center}
			\begin{tikzpicture}[scale=1.3]
				\draw[->, thick] (1.3,5)--(2.36,3.24) -- (6.3, 3.6) -- (7.9,4.7);
				\draw[<-] (8, 4) arc (90:150:4);
				\draw[-] (8,0) -- (5.65,5);
				\draw[-] (8,0) -- (8,4);
				\draw[->] (0, 4) arc (90:30:4);
				\draw[-] (0,0) -- (3.6,5);
				\draw[-] (0,0) -- (0,4);
				\draw[-] (0.5,4.45) -- (3.5,2.54);
				\draw[-] (8,4.3) -- (5,3.1);
				\draw[-, dotted] (-1,4) -- (9,4);
				\node at (-1, 4.2) {\small $\ell_0$};
				\draw[->] (2.4, 4) arc (360:305:0.5); \node at (2.55, 3.8) {\small $\tilde \theta$};
				\draw[->] (6.7, 4) arc (180:50:0.2); \node at (6.8, 4.35) {\small $\theta$};
				\draw[->] (8, 0.8) arc (90:115:0.8); \node at (7.8, 1)  {\small $\beta_2$};
				\draw[->] (0, 0.8) arc (90:56:0.8); \node at (0.3, 1)  {\small $\beta_1$};
				\draw[<-] (7.3, 4.3) arc (50:20:0.5); \node at (7.8, 4.25)  {\small $\pi-\theta_2$};
				\draw[<-] (3, 3.3) arc (17:-30:0.5); \node at (3.4, 3.05)  {\small $\pi-\theta_1$};
				\node at (6.3, 3.6) {\small $\bullet$};   \node at (6.2,3.25) {\small $r_{i,2}$};
				\node at (2.36,3.24) {\small $\bullet$};   \node at (2.4,2.8) {\small $r_{i,1}$};
			\end{tikzpicture}
			\caption{Relations between $\beta_1,\beta_2,\theta_1,\theta_2,\theta$ and $\tilde \theta$.}\label{fig:beta2}
		\end{center}
	\end{figure}
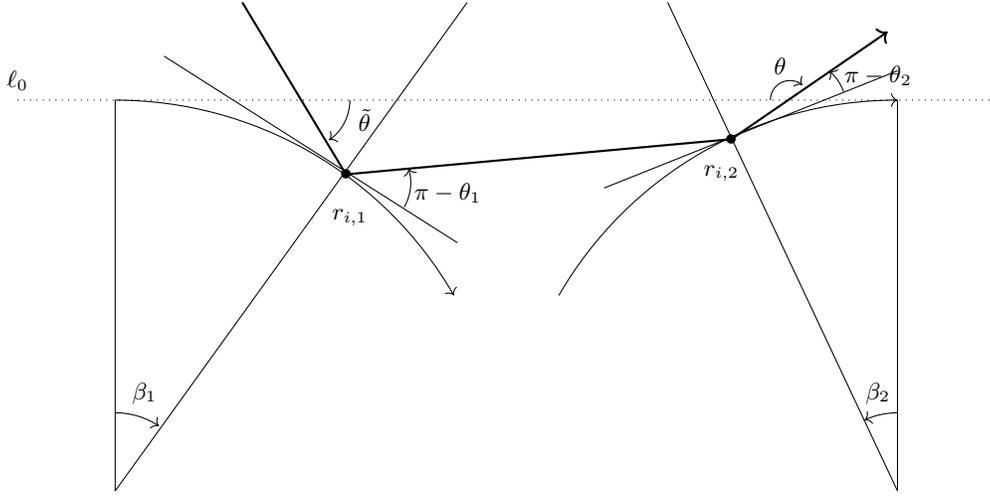

	\subsubsection{Estimating the $\|\cdot\|_\infty$ norm}
	
	\begin{lemma}\label{lem:continf}
		Assume $\sin \tilde\theta = \sin \Psi_{R_i}^{-1}(\theta) < \eta$.
		There exists $C_\infty>0$ (independent of $\theta$) so that for all $t\in\R$
		and $f\in BV$,
		$$
		\left\|(P_t-P_0)f|_{\{\sin\tilde\theta < \eta\}}\right\|_{\infty}
		\le C_\infty\, |t|\, \|f\|_{\infty}.
		$$
	\end{lemma}
	
	\begin{proof}
		Formula \eqref{eq:psm2} tells us that we have an extra factor $\nu(p(J_\xi,\tilde\theta))$ in each term of \eqref{eq:Pest}. We split the sum according to the type of trajectory inside the microstructure.
		If the inward trajectory (approaching $\ell_0 \cap \mic_i$ from the left)
		first hits the left cheek and then exits, then $\sin \theta \leq \sin \tilde \theta$. If the trajectory first hits the right cheek and then exits, then $\sin \tilde \theta \leq \sin \theta$.
		As a result, for fixed $\theta$, the ranges of $\sin \tilde \theta$ are adjacent subintervals of
		$[C_{\kappa}^{-1} \sin^2 \theta, \sin \theta
		]$ for the right cheek, and
		$[\sin \theta, C_\kappa \sqrt{\sin \theta}]$
		for the right cheek, and hence the boundaries of the sums for these cheeks in the computation below
		overlap only for $W/\sin\theta$.
		
		If there is a collision with both cheeks, then we 
		can still compare $\sin \tilde \theta$ and $\sin \theta$ according to Lemma~\ref{lem:SV},
		but we have $\sin \theta_{i,1} < \sin \tilde \theta$ and $\sin \theta_{i,2} < \sin \theta$,
		and also $|\Psi'_{R_i}|$ has an extra factor $\geq \tau_{i,1} \kappa_{i,1}/\sin\theta_{i,2} \geq \frac{\kappa_{\min} }{2 \sin\theta}$.
		This leads to three cases in the estimate of the derivative of \eqref{eq:W}
		and therefore three sums in the estimate
		$P(f \cdot  |e^{iX}-1|)(\theta)$, as follows:
		\begin{eqnarray*}
			\left| P(f \cdot  |e^{iX}-1|)(\theta)\right| &\ll&
			\frac{|t| W}{2\kappa_{\min} W}  \sum_{\xi = W \max\{ \frac{1}{\eta}, \frac{1}{\sqrt{C_\kappa \sin \theta}} \} }^{\frac{W}{\sin \theta}}
			\frac{ |\tan \tilde \theta|^2 \, | f(\tilde \theta)|\, |\cos \tilde \theta| }
			{ \frac{W \kappa_{\min}}{\sin^2 \tilde \theta}\, \sin \theta} \,
			\1_{\xi(\tilde\theta) = \xi}  \qquad \text{(left cheek)}
			\\
			&& + \ \frac{|t| W}{2\kappa_{\min} W}  \sum_{\xi = W \max\{ \frac{1}{\eta}, \frac{1}{\sin \theta} \} }^{\frac{W}{C_\kappa\sin^2 \theta}}
			\frac{ |\tan \tilde \theta|^2 \ f(\tilde \theta)\, |\cos\tilde \theta| }
			{\frac{W \kappa_{\min}}{\sin \theta \sin \tilde \theta} \,\sin \theta} \,
			\1_{\xi(\tilde\theta) = \xi}   \quad \qquad \text{(right cheek)}\\
			&& + \ \frac{|t| W}{2\kappa_{\min} W}  \sum_{\xi = W\max\{ \frac{1}{\eta}, \frac{1}{\sqrt{C_\kappa \sin \theta}} \} }^{\frac{W}{C_\kappa\sin \theta}}
			\frac{ |\tan \tilde \theta|^2 \, |f(\tilde \theta)| \, |\sin \tilde \theta| }
			{ \frac{2W \kappa^2_{\min}}{\sin^2 \tilde \theta \sin\theta}\, \sin \theta \,  |\tan \tilde \theta|} \,
			\1_{\xi(\tilde\theta) = \xi}   \quad \text{(double cheek)}   \\
			& = &
			\frac{|t| W^3}{\kappa_{\min}^2 \sin \theta}  \sum_{\xi = W \max\{ \frac{1}{\eta}, \frac{1}{\sqrt{C_\kappa \sin \theta}} \} }^{\frac{W}{\sin \theta}}
			\frac{ |\sin \tilde \theta|^4 \, | f(\tilde \theta) |}
			{W^4} \, \1_{\xi(\tilde\theta) = \xi}  \\
			&& + \ \frac{|t| W^2}{\kappa_{\min}^2}  \sum_{\xi = W \max\{ \frac{1}{\eta}, \frac{1}{\sin \theta} \} }^{\frac{W}{C_\kappa\sin^2 \theta}}
			\frac{ |\tan \tilde \theta|^3 \, |f(\tilde \theta)| }
			{W ^3} \1_{\xi(\tilde\theta) = \xi}  \\
			&& + \ \frac{|t| W^3}{2\kappa_{\min}^3}  \sum^{\frac{W}{C_\kappa \sin^2 \theta}}_{\xi = W \max\{ \frac{1}{\eta}, \frac{1}{\sqrt{C_\kappa \sin \theta}} \} }
			\frac{ |\tan \tilde \theta|^4 \, |f(\tilde \theta)| }
			{W^4} \, \1_{\xi(\tilde\theta) = \xi}  \\
			&\ll& \frac{|t| W^3 \|f\|_\infty}{\kappa_{\min}^2 \sin \theta}  \sum_{\xi = W \max\{ \frac{1}{\eta}, \frac{1}{\sqrt{C_\kappa \sin \theta}} \} }^{\frac{W}{\sin \theta}} \frac{1}{|\xi|^4}
			\ + \ \frac{|t| W^2 \|f\|_\infty}{\kappa_{\min}^2}  \sum_{\max\{ \xi_\eta , \frac{W}{\sin \theta} \} }^{W(C_\kappa^{-1}\sin \theta)^{-2}}
			\frac{1} {\xi^3} \\
			&&  + \ \frac{|t| W^3 \|f\|_\infty}{2\kappa_{\min}^3}  \sum_{\xi = W\max\{ \frac{1}{\eta}, \frac{1}{\sqrt{C_\kappa \sin \theta}} \} }^{\frac{W}{C_\kappa \sin^2 \theta}} \frac{1} {\xi^4} \\
			&\leq& \frac{|t|  \|f\|_\infty}{\kappa_{\min}^2}
			\left(\frac{C_\kappa \eta}{3}  + \frac{\eta^2}{2} + \frac{\eta^3}{3\kappa_{\min}} \right),
		\end{eqnarray*}
		so taking $C_\infty = \frac{\eta}{2\kappa_{\min}^2}\left(\frac{C_\kappa}{3}  + \frac{\eta}{2} + \frac{\eta^2}{3\kappa_{\min}} \right)$ gives the lemma.
	\end{proof}

	\subsubsection{Estimating the variation}\label{subsec:varsmall}

	\begin{lemma}\label{lem:ctvar}
		Assume $\sin \theta < \eta$.
		Let $f\in BV$. There exists $C_{Var}>0$ (independent of $\theta$) so that for all $t\in\R$,
		$$
		\Var_\theta\left((P_t-P_0)f|_{\{\sin\tilde\theta < \eta\}}\right)
		\le C_{Var}\, |t|\, \|f\|_{BV}.
		$$
	\end{lemma}
	
	\begin{proof}
		We first obtain a general bound for the variation
		starting from \eqref{eq:avadv}.
		\begin{align*}
			\Var_{\theta}&(Pf(e^{itX}-1)|_{\{ \sin \tilde \theta < \eta\} } ) \\
			&=  \Var_{\theta} \left( \sum_{|\xi| \geq \xi_\eta} \sum_{\ell \in \{ L,R,D\}} \int_{p(J_{\xi,\ell}, \tilde\theta)}
			\frac{f(\tilde \theta)}{|\Psi'_{R_i}(\tilde\theta)|} \frac{\sin \tilde \theta}{\sin \theta} (e^{itX(\tilde\theta)} - 1) \1_{J_{\xi,\ell}}(\tilde\theta) \, d\nu(R_i)
			\right) \\
			&\leq \sum_{|\xi| \geq \xi_\eta} \sum_{\ell \in \{ L,R,D\}}
			\Var_{\theta} \left(  \int_{p(J_{\xi,\ell}, \tilde\theta)} 
			\frac{f(\tilde \theta)}{|\Psi'_{R_i}(\tilde\theta)|} \frac{\sin \tilde \theta}{\sin \theta} (e^{itX(\tilde\theta)} - 1) \1_{J_{\xi,\ell}}(\tilde\theta) \, d\nu(R_i)
			\right).
		\end{align*}
		The integral is over $R_i$, not $\theta$, and for each $(\xi,\ell)$, the set $\Theta_{\xi,\ell} = \bigcup_{R_i} \Psi_{R_i}(J_{\xi,\ell})$ (where the union only runs over those $R_i \in [0,1]$ for which the branch $J_{\xi,\ell}$ actually exists)
		is an interval of length $O(\sqrt{1/|\xi|})$, according to Lemma~\ref{lem:SV}.
		This is too long for our purpose.
		
		However, for each pair $(\xi,\ell)$, and  $\tilde \theta$ with $\xi(\tilde\theta) = \xi$, the measure
		$\nu(p(J_{\xi,\ell},\tilde\theta)) \leq \frac{h^+ W}{2\kappa_{\min} (|\xi|-1)^2}$
		by Lemma~\ref{lem:p}. Also, the interval $p(J_{\xi,\ell},\tilde\theta)$ moves continuously in $\tilde\theta$.
		Therefore, if $\Theta_{\xi,\ell,k}$ is a sufficiently small
		neighbourhood of $\theta = \Psi_{R_i}(\tilde\theta)$,
		and we set
		$$
		p(\xi,\ell,k) = \Big\{ R_i \in [0,1] : R_i \in p(J_{\xi\,\ell},\tilde\theta) \text{ for some } \tilde\theta \text{ with }
		\Psi_{R_i}(\tilde\theta) \in \Theta_{\xi,\ell,k} \Big\},
		$$
		then we can assure that
		$\nu(p(\xi,\ell,k))$ is about four times as big, say
		$\frac{h^+ W}{\kappa_{\min} \xi^2} \leq \nu(p(\xi,\ell,k)) \leq \frac{2h^+ W}{\kappa_{\min} \xi^2}$.
		The derivative of the corresponding branch $\Psi_{R_i}$ is bounded away from zero, and therefore
		we can partition $\Theta_{\xi,\ell}$ into finitely many subintervals $\Theta_{\xi,\ell,k}$ (i.e., for $k$ in a finite index set $K_{\xi,\ell}$) such that
		$\nu(p(\xi,\ell,k)) \leq \frac{h^+ W}{\kappa_{\min} |\xi|^2}$
		for each $k \in K_{\xi,\ell}$. Then
		\begin{align}\label{eq:varsum}
			\Var_{\theta}  & \left(  \int_{p(J_{\xi,\ell}, \tilde\theta)}
			\frac{f(\tilde \theta)}{|\Psi'_{R_i}(\tilde\theta)|} \frac{\sin \tilde \theta}{\sin \theta} (e^{itX(\tilde\theta)} - 1) \1_{J_{\xi,\ell}}(\tilde\theta) \, d\nu(R_i)
			\right) \nonumber \\
			& = \
			\sum_{k \in K_{\xi,\ell}}  \Var_{\theta} \left(  \int_{p(\xi,\ell,k)}
			\frac{f(\tilde \theta)}{|\Psi'_{R_i}(\tilde\theta)|} \frac{\sin \tilde \theta}{\sin \theta} (e^{itX(\tilde\theta)} - 1) \1_{\Theta_{\xi,\ell,k} } \circ \Psi_{R_i}(\tilde\theta)  \, d\nu(R_i)
			\right) \nonumber \\
			&\leq \
			\frac{2h^+ W}{\kappa_{\min} \xi^2} \sum_{k \in K_{\xi,\ell}}
			\sup_{R_i \in p(\xi,\ell,k)} \Var_{\theta} \left(
			\frac{f(\tilde \theta)}{|\Psi'_{R_i}(\tilde\theta)|} \frac{\sin \tilde \theta}{\sin \theta} (e^{itX(\tilde\theta)} - 1) \, \1_{\Theta_{\xi,\ell,k} } \circ \Psi_{R_i}(\tilde\theta)
			\right) \nonumber \\
			& \leq \frac{C_p W}{\xi^2} \sum_{k \in K_{\xi,\ell}} \
			\sup_{R_i \in [0,1]} \Var_{\theta} \left(
			\frac{f(\tilde \theta)}{|\Psi'_{R_i}(\tilde\theta)|} \frac{\sin \tilde \theta}{\sin \theta} (e^{itX(\tilde\theta)} - 1) \1_{\Theta_{\xi,\ell,k}} \circ \Psi_{R_i}(\tilde\theta) \right),
		\end{align}
		for some $C_p$ independent of $\xi$ and $W$. Let $J_{\xi,\ell,k} = \{ \tilde\theta \in J_{\xi,\ell} : \Psi_{R_i}( \tilde\theta) \in \Theta_{\xi,\ell,k} \}$.
		These intervals depend on $R_i$, but in order not to overload the notation even more, we will suppress this dependence.
		
		The map $\Psi_{R_i}|_{J_{\xi,\ell}}$ is monotone for each $|\xi| \geq \xi_\eta, \ell \in \{L,R,D\}$. So, if we split the above into parts depending on $\tilde \theta \in J_{\xi,\ell}$ and $\theta = \Psi_{R_i}(\tilde\theta)$, we can look at the variation in $\tilde\theta$:
		\begin{align}\label{eq:var1}
			\Var_{\theta} & \left(
			\frac{f(\tilde \theta)}{|\Psi'_{R_i}(\tilde\theta)|} \frac{\sin \tilde \theta}{\sin \theta} (e^{-tX(\tilde\theta)} - 1) \1_{J_{\xi,\ell,k}}(\tilde\theta)  \right) \nonumber \\
			& =\
			\Var_{\theta} \left( \frac{1}{\sin\theta} |_{\Theta_{\xi,\ell,k}} \right) \sup_{\tilde\theta \in J_{\xi,\ell,k}}
			\left| \frac{f(\tilde \theta)}{|\Psi'_{R_i}(\tilde\theta)|} \sin \tilde \theta (e^{itX(\tilde\theta)} - 1) \right| \nonumber \\
			& \qquad + \ \sup_{\theta \in \Theta_{\xi,\ell,k}}\frac{1}{\sin\theta}
			\Var_{\tilde \theta \in J_{\xi,\ell,k}}
			\left( \frac{f(\tilde \theta)}{|\Psi'_{R_i}(\tilde\theta)|} \sin \tilde \theta (e^{itX(\tilde\theta)} - 1)  \right)\nonumber\\
			& = \ I_{\xi,\ell,k}(f)+II_{\xi,\ell,k}(f),
		\end{align}
		and we estimate the sums of $I_{\xi,\ell,k}$ and $II_{\xi,\ell,k}$ over $k \in K_{\xi,\ell}$.
		
		We start with $I_{\xi,\ell,k}$ in \eqref{eq:var1}.
		As $\theta \mapsto \frac{1}{\sin \theta}$ is monotone on $\Psi_{R_i}(J_{\xi,\ell})$, and keeping in mind that the sum of variations of a continuous function over adjacent intervals is the variation over the union of those intervals,
		we have
		$$
		\sum_{k \in K_{\xi,\ell}} \sup_{R_i \in p(\xi,\ell,k)} \Var_{\theta}\left(\frac{1}{\sin\theta}|_{\Psi_{R_i}(J_{\xi,\ell})} \right) \leq \sup_{\theta \in \Psi_{R_i}(J_{\xi,\ell})} \frac{1}{\sin\theta}.
		$$
		Hence, the first term in \eqref{eq:var1} contains only suprema, and can be bounded by $|t| \kappa_{\min}^{-2} \| f \|_\infty$ using the computation similar to the one used in Lemma~\ref{lem:continf}. More precisely, recall from~\eqref{twoDifs}
		that $\frac{1}{|\Psi'_{R_i}(\tilde\theta)|} \le \frac{\sin \thetain_i  \max\{\sin \thetaout_i,  \sin \thetain_i\}}{W \kappa_{i,1}}$.
		Considering the worst case (i.e., left cheek collision)
		with
		$\max\{\sin \thetaout_i,  \sin \thetain_i\} = \sin \thetain_i = \sin \tilde\theta$,
		we obtain that
		$\frac{1}{|\Psi'_{R_i}(\tilde\theta)|} \le C \frac{\sin^2\tilde\theta}{W \kappa_{\min}}$.
		Recalling that $\sin\tilde\theta\sim W/|\xi|$ and $\frac{1}{\sin\theta}\le C_\kappa^2\frac{\xi^2}{W^2}$
		from Lemma~\ref{lem:SV}, we have:
		\begin{align*}
			\sum_{k \in K_{\xi,\ell}} \sup_{R_i \in p(\xi,\ell,k)} I_{\xi,\ell,k}(f)\le C_\kappa^2 \, |t|\,\|f\|_\infty\frac{\xi^2}{W^2}\frac{W^2}{\xi^2} = C_\kappa^2 \, |t|\,\|f\|_\infty.
		\end{align*}
		Thus, the sum in~\eqref{eq:varsum} coming from $I$
		satisfies (recall $\xi_\eta \sim W/\eta$),
		\begin{equation}\label{eqI}
			\sum_{|\xi| > \xi_\eta} \sum_{\ell \in \{ L,R,D\}} \sum_{k \in K_{\xi,\ell}}
			I(\xi,f) \le C_\kappa^2\, |t| \, \|f\|_\infty\sum_{|\xi| \geq \xi_\eta} \sum_{\ell \in \{ L,R,D\}} \frac{C_p W}{\xi^2} = O\left( |t|\,\|f\|_\infty \right).
		\end{equation}
		
		Regarding $II_{\xi,\ell,k}$ in \eqref{eq:var1}, we first separate $f$:
		\begin{eqnarray}\label{eq:II}
			II_{\xi,\ell,k}(f) &\leq&
			\sup_{\theta \in \Theta_{\xi,\ell,k}} \frac{1}{\sin\theta} \Var(f|_{J_{\xi,\ell,k}})\sup_{\tilde\theta \in J_{\xi,\ell}}
			\left| \frac{\sin \tilde\theta \, (e^{itX(\tilde\theta)}-1)}{|\Psi'_{R_i}(\tilde\theta)| } \right| \nonumber \\
			&& + \
			\sup_{\theta \in \Theta_{\xi,\ell,k}}  \frac{\| f \|_\infty}{\sin\theta} \Var_{\tilde\theta \in J_{\xi,\ell,k}} \left( \frac{\sin \tilde\theta \, (e^{itX(\tilde\theta)}-1)}{|\Psi'_{R_i}| }  \right).
		\end{eqnarray}
		
		For the remaining terms, we first note that  $x \mapsto e^{itx}-1$ is smooth and sufficiently ``monotone'' so that $\sum_{k \in K_{\xi,\ell}} \Var_{\tilde\theta \in J_{\xi,\ell,k} }(e^{itX(\tilde\theta)}-1)
		\leq |t| \sum_{k \in K_{\xi,\ell}} \Var_{\tilde\theta \in J_{\xi,\ell,k} } X(\tilde\theta) \leq |t| \sup_{\tilde\theta \in J_{\xi,\ell}} \frac{W}{|\tan \tilde\theta|}$.
		Therefore,
		\begin{eqnarray}\label{vraspl}
			\sum_{k \in K_{\xi,\ell}} \sup_{R_i \in p(\xi,\ell,k)} \Var_{\tilde\theta \in J_{\xi,\ell,k} }\left( \frac{\sin \tilde\theta \, (e^{it}-1)}{|\Psi'_{R_i}| } \right) &\le& \\
			\sum_{k \in K_{\xi,\ell}} 	|t| \sup_{\tilde\theta \in J_{\xi,\ell}} \frac{W}{|\tan \tilde\theta|}\sup_{R_i \in p(\xi,\ell,k)}\left( \frac{\sin \tilde\theta}{|\Psi'_{R_i}(\tilde\theta)| } \right)
			+ |t| \sup_{\tilde\theta \in J_{\xi,\ell}} \!\!\! \ &
			\frac{W}{|\tan \tilde\theta|} &
			\sup_{R_i \in p(\xi,\ell,k)}  \Var_{\tilde\theta \in J_{\xi,\ell,k} }\left( \frac{\sin \tilde\theta}{|\Psi'_{R_i}| } \right).  \nonumber
		\end{eqnarray}
		We can merge the intervals $J_{\xi,\ell,k}$ over $k \in K_{\xi,\ell}$ again and apply  Corollary~\ref{cor:Q} for the remaining sum over the variations in
		this expression.
		Using again the estimates listed before~\eqref{eqI},
		continuing from~\eqref{vraspl}, we have
		that the following holds for some $C_0,C_1>0$:
		\begin{align*}
			\sum_{k \in K_{\xi,\ell}} \sup_{R_i \in p(\xi,\ell,k)} \Var_{\tilde\theta \in J_{\xi,\ell,k}} \left( \frac{\sin \tilde\theta \, (e^{itX}-1)}{|\Psi'_{R_i}| }  \right)& \le
			C_0 |t| \frac{W^2}{\xi^2} + C_1 |t| \frac{1}{|\xi|^{3/2}}.
		\end{align*}
		This together with~\eqref{eq:II} (recalling $\frac{1}{\sin\theta}\le C_\kappa\frac{\xi^2}{W^2}$ and $\xi_\eta \sim W/\eta$) gives:
		\begin{align*}
			\sum_{|\xi| > \xi_\eta} & \sum_{\ell \in \{ L,R,D\}}
			\sum_{k \in K_{\xi,\ell}} \sup_{R_i \in p(\xi,\ell,k)}
			J_{\xi,\ell,k}(f) \\
			\le\ & C\,|t|\,\|f\|_{BV}\sum_{|\xi| > \xi_\eta} \sum_{\ell \in \{L,R,D\}}
			\frac{C_p W}{\xi^2} \frac{C_\kappa \xi^2}{W^2}
			\left(  C_0 |t| \frac{W^2}{\xi^2} + C |t| \frac{1}{|\xi|^{3/2}} \right)
			= \ O\left(|t|\|f\|_{BV}\right).
		\end{align*}
		This together with~\eqref{eqI} implies that
		$\Var_\theta\left((P_t-P_0)f\right)=O\left(\frac{|t|\|f\|_{BV}}{W}\right)$, as required.
	\end{proof}
	
	An immediate consequence of Lemmas~\ref{lem:continf}
	and~\ref{lem:ctvar} is
	
	\begin{corollary}\label{cor:ctvar}Assume $\sin\theta < \eta$.
		Let $f\in BV$. There exists $C>0$ (independent of $\theta$) so that
		$\left\|(P_t-P_0)f|_{\{\sin\tilde\theta < \eta\}}\right\|_{BV} \le C\,|t|\, \|f\|_{BV}$ for all $t\in\R$.
	\end{corollary}

	\subsection{Continuity estimates when $\sin\tilde\theta\ge \eta$}
	\label{subsec:contaw}
	As in the proof of Lemma~\ref{lem:contaw} below, the averaging plays no role when $\sin\tilde\theta\ge \eta$.
	Throughout this paragraph we shall exploit the formula for transfer operator $P_{R_i}$ defined in~\eqref{eq:PR}.

	\begin{lemma}\label{lem:contaw} 
		There exists $C>0$ so that
		$\left\|(P_t-P_0)f \Big|_{\{\sin \tilde\theta \geq \eta\}} \right\|_{BV}\le C|t|\|f\|_{BV}$.
	\end{lemma}

	\begin{proof}
		We display the argument for bounding the variation. The argument for the $\|\cdot\|_\infty$ norm is simpler and omitted.
		
		If $\theta \mapsto \Psi_{R_i}^{-1}(\theta) = \tilde \theta \in J_{\xi,\ell}$
		refers to a single monotone branch of $\Psi_{R_i}^{-1}$,
		then $\Var_\theta (f \circ \Psi_{R_i}^{-1}) = \Var_{ \theta}(f|_{I_{\xi,\ell}})$.
		Recall from Lemma~\ref{lem:Lambda} that
		$\Lambda_\eta$ indicates the collection of branches of $\Psi^{-1}_{R_i}$ for $\sin \tilde\theta \geq \eta$.
		
		With these specified, writing again $\tilde\theta = \Psi_{R_i}^{-1}(\theta)$ and using $\Var(f \cdot g) = \| f \|_\infty \Var(g) + \| g \|_\infty \Var(f)$ multiple times, we compute that
		\begin{align*}
			\Var_\theta & \left( P_{R_i} \left(e^{itX}-1\right)f \big|_{\{\sin \tilde\theta \geq \eta\} } \right) \\
			\leq	&\  \Var_\theta\left(\frac{1}{\sin \theta} \Big|_{\{ \sin \tilde\theta \geq \eta\} }\right)
			\sum_{\ell \in \Lambda_\eta} \sup_{\tilde\theta \in J_\ell}
			\frac{f(\tilde\theta) \cdot \sin(\tilde \theta) (e^{itX}-1)}
			{ |\Psi'_{R_i}(\tilde\theta)| }  \Big|_{\{\sin \tilde\theta \geq \eta\}} \\
			& \ + \
			\sup_{\sin\tilde \theta \geq \eta} \left( \frac{1}{\sin \theta} \right) \sum_{\ell \in \Lambda_\eta}
			\Var_{\tilde\theta \in J_\ell} \left( \frac{f(\tilde\theta) \cdot \sin(\tilde \theta) (e^{-itX}-1) }
			{|\Psi'_{R_i}(\tilde\theta)| } \right)  \\
			\leq & \
			\frac{|t| W C_\kappa^2}{\eta^2} \| f \|_\infty
			\sum_{\ell \in \Lambda_\eta}  \left\| \frac{1 }{ | \Psi'_{R_i}(\tilde \theta) | }\right\|_\infty
			+ \frac{C_\kappa^2}{\eta^2} \sum_{\ell \in \Lambda_\eta}
			\Var_{\tilde\theta \in J_\ell}\left( \frac{f(\tilde\theta) \cdot \sin(\tilde \theta) (e^{-itX}-1) }
			{|\Psi'_{R_i}(\tilde\theta)| } \right)
			\\
			\leq & \
			\frac{|t| W C_\kappa^2}{\eta^2}
			\sum_{\ell \in \Lambda_\eta}  \left\| \frac{1 }{ | \Psi'_{R_i}(\tilde \theta) | }\right\|_\infty  \Big( 3\| f \|_\infty + \Var(f) \Big)
			+ \frac{|t| W C_\kappa^2}{\eta^2} \sum_{\ell \in \Lambda_\eta}
			\Var_{\tilde\theta \in J_\ell}\left( \frac{1}
			{|\Psi'_{R_i}(\tilde\theta)| } \right)
			\\
			& \leq \frac{|t| W C_\kappa^2}{\eta^2}  \Big( 3\| f \|_\infty + \Var(f) \Big) C_\eta +
			\frac{|t| W C_\kappa^2}{\eta^2} \| f \|_\infty C_\Psi
			\leq \frac{C' W}{\eta^2} |t|\|f\|_{BV},
		\end{align*}
		for some $C'>0$. Here we used that $\Lambda_\eta$ pertains to at most $C_\eta W$ branches (see from Lemma~\ref{lem:Lambda}), and Lemma~\ref{lem:VarPsi} to get the bound
		$ \sum_{\ell \in \Lambda_\eta}
		\Var_{\tilde\theta \in J_\ell}\left( \frac{1}
		{|\Psi'_{R_i}(\tilde\theta)| } \right) \leq C_\Psi$
		for some $C_\Psi> 0$.
		The desired continuity estimate for the variation of the averaged operator follows immediately.
	\end{proof}
	
	\subsection{Proof of Proposition~\ref{prop:cont}}
	
	This follows at once from Corollary~\ref{cor:ctvar} and Lemma~\ref{lem:contaw} with $C_{BV} = C_\infty + C'/\sqrt{\eta}$.

	\section{Spectral properties for the averaged operator}\label{sec:LY}
	
	Unlike in previous literature on random dynamical systems (see~\cite{ANS}
	and references therein), in the present
	set up we have uniform expansion (that is, not just in average), but to deal with the variation in $\theta$ of the transfer operator in the case that $\sin\theta<\eta$, we will heavily exploit the averaging (in a similar manner as in Section~\ref{sec:continuity}), see Section~\ref{subsec:lycl}.
	For $\sin\theta\ge \eta$, averaging plays no role (again, similar to the continuity estimate in Section~\ref{sec:continuity}), see Section~\ref{subsec:contaw}.
	
	The result below gives the required Lasota-Yorke inequalities in BV.
	
	\begin{prop}\label{prop:ly}
		There exist $\alpha\in (0,1)$ and $C_1, C_2>0$ so that for all $n\ge 1$ and all $f\in BV$,
		$$
		\|P^nf\|_{BV} \le \alpha^n\|f\|_{BV}+C_1\|f\|_\infty \quad \text{ and } \quad \|P^nf\|_\infty \le C_2\|f\|_\infty.
		$$
	\end{prop}
	The proof of this result is provided in 
	Sections~\ref{subsec:lycl}--\ref{subsec:ly}.
	
	\subsection{Spectral decomposition of $P$}
	\label{subsec:spde}
	
	Proposition~\ref{prop:ly} together with a classical result~\cite{ITM}
	implies that when regarded as an operator in BV, for $n\ge 1$,
	$
	P^n=\sum_i\lambda_i^n \Pi_i + Q^n,
	$
	where $\lambda_i$ are eigenvalues of modulus $1$, $\Pi_i$ are finite-rank projectors onto the associated
	eigenspaces, and $Q$ is a bounded operator with a spectral radius strictly less than $1$. Also, there are only finitely many eigenvalues on the unit circle, and all $\lambda_i$ are roots of unity. (This type of decomposition for the perturbed averaged operator $P_t$ would be enough for the proof of Theorem~\ref{thm:main}.)

	Moreover, we can also ensure that $1$ is a simple isolated eigenvalue in the spectrum of $P$. To do so, we will employ the correspondence between properties of random dynamical systems and associated Markov chains
	as in, for instance,~\cite{Kif, ANS}.

	Following a similar notation as in~\cite[Sections 2 and 4]{ANS},
	we note that the averaged Koopman operator $U$ acts on functions defined on $(0,\pi)$ via $Uf=\int_0^1 f(\Psi_{R_i}) d\nu(R_i)$.
	In particular, $U$ corresponds to a transition probability matrix on $(0,\pi)$
	defined in general for $n$-steps by
	\begin{align*}
		(U^n \1_A)(\theta)=\nu^{\otimes \Z} \left((R_j)_{j \in \Z}\in [0,1]^{\Z}: \Psi_{R_{n-1}} \circ \cdots \circ \Psi_{R_0}(\theta)\in A\right), \quad A\in\cA,
	\end{align*}
	where 
	$\cA$ is the $\sigma$-algebra of $\mu$-measurable sets on $(0,\pi)$.
	
	Let
	\begin{align}\label{defMC}
		Y_n((R_i)_{i \in \Z}, \theta)=\Psi_{R_{n-1}} \circ \cdots \circ \Psi_{R_0}(\theta),\quad \theta \in A\in\cA,\quad (R_0, \dots, R_{n-1})\in [0,1]^n.
	\end{align}
	Then $(Y_n)_{n\ge 1}$ defines a homogeneous Markov chain on state space
	$((0,\pi),\cA)$. The transition operator (probability matrix) is given by $U$.

	Recall that $\mu$ is an invariant measure for $\Psi_{R_i}$.
	Since $\mu(A)=\int_0^1\mu(\Psi_{R_i}^{-1} (A))\, d\nu(R_i)$ for each $A\in\cA$,
	$\mu$ is a stationary measure for the associated Markov chain
	and $\mu U=\mu$.
	As clarified in Lemma~\ref{lem:erg} below,
	the Markov chain $(Y_n)_{n\ge 1}$ is aperiodic. Thus, $\mu$ is the unique stationary measure for this Markov chain and thus the unique left eigenvector of $U$ with eigenvalue $1$, which is simple.
	Recalling Remark~\ref{rmk:integr}, we see that
	the averaged operator $P$ is the dual, or adjoint, operator of $U$, i.e., $\int_0^\pi P f\, g\,d\mu=\int_0^\pi f\, Ug\, d\mu$.

	It follows that the constant function $1$ is the unique eigenfunction of modulus $1$ for $P$. Indeed, if $f \not\equiv 1$ were a fixed point of 
	the average transfer operator: $f = Pf := \int_{[0,1]^\Z} P_{(R_i)_{i \in \Z}} f \, d\nu^{\otimes \Z}$, that is, if $1$ is not a simple eigenvalue of $P$, then,
	using duality for an arbitrary $g:(0,\pi) \to \R$, 
	\begin{eqnarray*}
		\int_0^\pi \int_{[0,1]^\Z} f \cdot g \circ \Psi_{R_i} \, d\nu^{\otimes \Z} \, d\mu
		&=& \int_0^\pi \int_{[0,1]^\Z} P_{(R_i)_{i \in \Z}} f \cdot g \, d\nu^{\otimes \Z} \, d\mu \\
		&=& \int_0^\pi \left( \int_{[0,1]^\Z} P_{(R_i)_{i \in \Z}} f \, d\nu^{\otimes \Z} \right) \cdot g \, d\mu\\
		&=& \int_0^\pi f \cdot g \, d\mu 
		= \int_0^\pi f \circ \Psi_{R_i} \cdot g \circ \Psi_{R_i} \, d\mu\\
		&=& \int_0^\pi \int_{[0,1]^\Z} f \circ \Psi_{R_i} \cdot g \circ \Psi_{R_i}    \, d\nu^{\otimes \Z} \, d\mu.
	\end{eqnarray*}
	Since $g$ is arbitrary,
	this shows that 
	$$
	\int_{[0,1]^\Z} f(x) \circ \Psi_{R_i} \cdot g \circ \Psi_{R_i}(x)    \, d\nu^{\otimes \Z} =
	\int_{[0,1]^\Z}  f(x) \cdot g \circ \Psi_{R_i}(x) \, d\nu^{\otimes \Z}.
	$$
	for $\mu$-a.e.\ $x \in (0,\pi)$.
	The special case $g \equiv 1$ gives
	$$
	Uf(x) :=  \int_{[0,1]^\Z} f \circ \Psi_{R_i}(x)  \, d\nu^{\otimes \Z} = \int_{[0,1]^\Z}  f(x) \, d\nu^{\otimes \Z} = f(x) \quad \mu\text{-a.s.},
	$$
	but this contradicts that $U$ has a unique fixed point, i.e., it contradicts the uniqueness of the stationary measure.
	
	Hence $1$ is a simple eigenvalue of $P$. If $\lambda_i$ were another eigenvalue of the unit circle, say $\lambda_i^k = 1$, then we repeat the argument with $U^k$ and $P^k$.
	This would imply that the eigenvalue $1$ would not be simple for this iterate,
	contradicting the aperiodicity of $U$.	
	Thus, $1$ is the only eigenvalue on the unit circle, and
	\begin{align*}
		P^n=\Pi+ Q^n,\quad \Pi f=\int_0^\pi f\, d\mu, \quad \|Q^n f\|_{BV}\le \delta^n,\text{ for some }\delta\in (0,1).
	\end{align*}
	An immediate consequence is exponential decay of correlation for $f\in BV$ and $g\in L^\infty$, in the sense that
	$\left|\int_0^\pi f\, U^n g\, d\mu-\int_0^\pi f\, d\mu\int_0^\pi g\, d\mu\right|\le C \delta^n \|f\|_{BV}\, \|g\|_{L^\infty}$, for some $C>0$ and $\delta \in (0,1)$ independent of $f, g$, and $n \geq 1$.

	\subsection{Aperiodicity of the associated Markov chain}
	\label{sec:erg}
	
	Let $\cI^n_i(A) = \{ \Psi_{R_{i+n-1}} \circ \cdots \circ \Psi_{R_i}(\thetaout_{i-1}) :
	\theta \in A, (R_i, \dots, R_{i+n-1}) \in [0,1]^n\}$.
	
	\begin{lemma}\label{lem:I}
		For every $\theta \in (0,1)$ and $i \in \Z$, we have
		$$
		\bigcup_{n \geq 0} \cI^n_i(\{\thetaout_{i-1} \}) = (0,\pi).
		$$
	\end{lemma}

	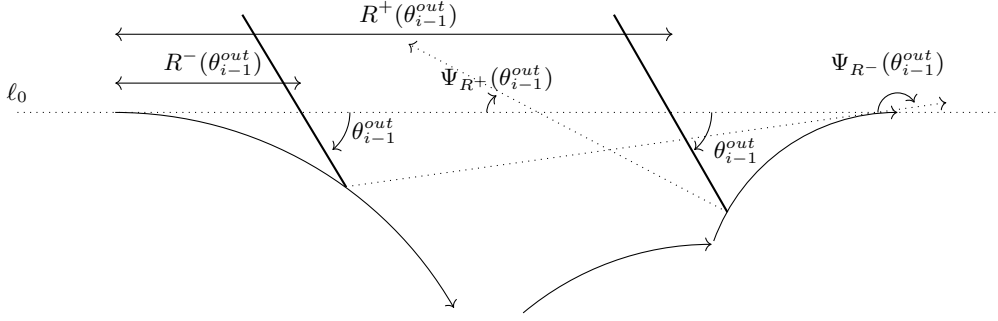
\begin{figure}[ht]
		\begin{center}
			\begin{tikzpicture}[scale=1.3]
				\draw[-, thick] (1.3,5)--(2.36,3.24);
				\draw[->, dotted] (2.36,3.24) -- (8.5, 4.1);
				\draw[-, thick] (5.1,5)--(6.26,2.98);
				\draw[->, dotted] (6.26,2.98) -- (3, 4.7);
				\draw[<->] (0,4.3) -- (1.9, 4.3); \node at (1, 4.5) {\small $R^-(\thetaout_{i-1})$};
				\draw[<->] (0,4.8) -- (5.7, 4.8); \node at (3, 5) {\small $R^+(\thetaout_{i-1})$};
				\draw[<-] (8, 4) arc (90:151:2);
				\draw[<-] (6.23, 2.95) arc (90:130:3);
				\draw[->] (0, 4) arc (90:30:4);
				\draw[-, dotted] (-1,4) -- (9,4);
				\node at (-1, 4.2) {\small $\ell_0$};
				\draw[->] (2.4, 4) arc (360:310:0.5); \node at (2.65, 3.8) {\small $\thetaout_{i-1}$};
				\draw[->] (6.1, 4) arc (360:310:0.5); \node at (6.35, 3.6) {\small $\thetaout_{i-1}$};
				\draw[->] (7.8, 4) arc (180:30:0.2); \node at (7.9, 4.5) {\small $\Psi_{R^-}(\thetaout_{i-1})$};
				\draw[->] (3.8, 4) arc (180:120:0.2); \node at (3.9, 4.35) {\small $\Psi_{R^+}(\thetaout_{i-1})$};
			\end{tikzpicture}
			\caption{The definition of $R^+(\thetaout_{i-1})$ and $R^-(\thetaout_{i-1})$ illustrated.}\label{fig:Rpm}
		\end{center}
	\end{figure}

	\begin{proof}
		For $\theta \in (0,\pi)$, let
		$$
		R^+(\theta) =
		\inf_{R \leq 1} \{ R : \text{the random trajectory in } \mic_i
		\text{ hits the right cheek and then leaves } \mic_i\}
		$$
		and
		$$
		R^-(\theta) =
		\sup_{R \geq 0} \{ R : \text{the random trajectory in } \mic_i
		\text{ hits the left cheek and then leaves } \mic_i\},
		$$
		see Figure~\ref{fig:Rpm}.
		This means that at $R^+$, the random trajectory either hits a corner point of the right cheek, or, after a collision with the right cheek, has a grazing collision with the left cheek,
		and similar for $R_-$.
		By convexity of the right and left cheek, for each $\theta \in (0,\pi)$,
		\begin{equation}\label{eq:Rpm}
			\lim_{R \searrow R^-(\theta)} \Psi_R(\theta) < \theta
			\quad  \text{ and } \quad
			\lim_{R \nearrow R^+(\theta)} \Psi_R(\theta) > \theta.
		\end{equation}
		Now define recursively
		$$
		\alpha_k =
		\begin{cases}
			\thetaout_{i-1} & \text{ if } k = 0,\\
			\lim_{R \searrow R^-_{i+k}(\alpha_{k-1})}\Psi_R(\alpha_{k-1}) & \text{ if } k \geq 1,
		\end{cases}
		\quad 
		\beta_k =
		\begin{cases}
			\thetaout_{i-1} & \text{ if } k = 0,\\
			\lim_{R \nearrow R^+_{i+k}(\beta_{k-1})}\Psi_R(\beta_{k-1}) & \text{ if } k \geq 1.
		\end{cases}
		$$
		Then $(\alpha_k)_{k \geq 0}$ is decreasing, and since
		it is bounded below by $0$, there is a limit $L$, which is a fixed point of the operation
		$\theta \mapsto \lim_{R \searrow R^+(\theta)} \Psi_R(\theta)$.
		This means by \eqref{eq:Rpm} that $L = 0$.
		The same argument shows that $(\beta_k)_{k \geq 0}$ is increasing to the limit $\pi$.
		
		Note that $(\alpha_1,\beta_1) = \{ \Psi_{R_i}(\thetaout_{i-1}) : R_i^- < R_i < R_i^+ \}$,
		and for larger values of $i$, there are subsets $G_j \subset [0,1]$ for $R_{i+j-1}$ so that
		$(\alpha_k,\beta_k) = \{ \Psi_{R_{i+k-1}} \circ \dots \circ \Psi_{R_i}(\thetaout_{i-1}) :
		R_{i+j} \in G_j, 0 \leq j < k \}$.
		If follows that
		$$
		\bigcup_{n \geq 0} \cI^n_i(\thetaout_{i-1})
		\supset \bigcup_{k\geq 0} (\alpha_k, \beta_k) = (0,\pi),
		$$
		as required.
	\end{proof}

	\begin{lemma}\label{lem:erg}
		The Markov chain $(Y_n)_{n\ge 1}$ defined in~\eqref{defMC} is aperiodic.
	\end{lemma}
	
	\begin{proof}
		We prove aperiodicity by showing that $(Y_n)_{n \geq 1}$
		is indecomposable of all orders.
		Let $n\in \N$ be given. For indecomposability of order $n$, it is sufficient to prove for any $A \in \cA$ with $\mu(A)>0$ and which satisfies $U^n\1_A(\theta) = 1$ for $\mu$-a.e $\theta \in A$, that $\mu(A) = 1$, see Definition 7.14 from Breiman's book \cite{brei}. Let such an $A$ be given and take $\theta \in A$. Then
		$$
		(U^n\1_A)(\theta) = \nu^{\otimes \Z} \left((R_j)_{j \in \Z}\in [0,1]^{\Z}: \Psi_{R_{n-1}} \circ \cdots \circ \Psi_{R_0}(\theta)\in A\right) = 1,
		$$
		so for $\nu^{\otimes\Z}\times \mu$-a.e.\ $((R_i)_{i\in\Z},\theta) \in \left[0,1\right]^{\Z}\times A$ we get
		$$T^{n}((R_i)_{i\in\Z},\theta) \in \left[0,1\right]^{\Z}\times A.$$
		Thus, for $\mu$-a.e. $\theta\in A$ we have $\mathcal{I}^n(\{\theta\}) \subseteq A \cup N$ for some set $N \in \cA$ with $\mu(N) = 0$, and likewise for any $k\in\N$ we get
		$$\mathcal{I}^{nk}(\theta) \subseteq A \cup N_k,$$
		for some $N_k \in \cA$ with $\mu(N_k) = 0$. We have that $\mathcal{I}^{nk}(\theta)$ is increasing in $k$ since $\Psi(0,\theta) = \Psi(1,\theta) = \theta$, and this implies that $\mathcal{I}^k(\theta) \subset \mathcal{I}^{nk}(\theta)$. Therefore, using Lemma~\ref{lem:I}, we get
		$$
		(0,\pi) = \bigcup_{k \geq 1}\mathcal{I}^{k}(\theta) \subseteq \bigcup_{k \geq 1}\mathcal{I}^{nk}(\theta) \subseteq A \cup \left(\bigcup_{k \geq 1}N_k\right) \subseteq (0,\pi),
		$$
		and this implies that $\mu(A) = 1$.
	\end{proof}

	\subsection{Estimating $\|Pf\|_{BV}$ when $\sin\tilde\theta < \eta$}
	\label{subsec:lycl}

	We start with $\|Pf\|_{BV}$ of the average transfer operator $P$ defined  in~\eqref{avTran}.
	
	\begin{lemma}\label{lem:varpfsm}
		Assume that $\sin\tilde\theta<\eta$. Then there exists $\alpha\in (0,1)$ so that for all $f\in BV$,
		$\|Pf\|_{BV} \le \alpha\|f\|_{BV}$.
	\end{lemma}

	\begin{proof}
		\textbf{Estimating $\Big\|Pf\Big|_{\{\sin\tilde\theta < \eta\}} \Big\|_\infty$.}
		
		As in Section~\ref{sec:continuity}, averaging over $R_i$ means that 
		in the formula for transfer operator $P_{R_i}$ defined in~\eqref{eq:PR}
		we multiply with $\nu(p(J_{\xi,\ell},\tilde\theta))$.
		Recalling that $\sin\tilde\theta<\eta$ and proceeding similarly to~\eqref{eq:psm2},
		
		\begin{align}\label{eq:smp}
			\left\| Pf \Big|_{\{\sin \tilde\theta < \eta\}} \right\|_\infty\le
			\sup_\theta \sum_{\xi = W\max\{ \frac{1}{\eta}, \frac{1}{C_\kappa \sqrt{\sin \theta}} \} }^{\frac{W}{(C_\kappa \sin \theta)^2}}
			\left\|\frac{ f(\tilde \theta) \sin \tilde \theta}
			{|\Psi'_{R_i}(\tilde \theta)| \sin \theta}\right\|_\infty\, \nu(p(\xi, \tilde\theta))
			\1_{\xi(\tilde\theta) = \xi}.
		\end{align}
		Recall that the estimate for the derivative is given
		in~\eqref{eq:W}.
		Recall from Lemma~\ref{lem:p} that in the worst case scenario (left cheek collision), $\nu(p(J_{\xi,\ell}, \tilde\theta)) \le \frac{h^+ \, |\tan \tilde \theta|^2}{2\kappa_{\min} W}$.
		Also, recall that $\sin\tilde\theta\sim W/|\xi|$.
		Putting these together and using~\eqref{eq:smp}, we obtain
		
		\begin{align*}
			\|Pf \Big|_{\{\sin\tilde\theta < \eta\}} \|_{\infty}&\le
			\sup_\theta \|f\|_\infty \frac{1}{\sin\theta}
			\sum_{\xi = W \max\{ \frac{1}{\eta}, \frac{1}{\sqrt{C_\kappa \sin \theta}} \}}^{\frac{W}{(C_\kappa \sin \theta)^2}} \frac{\sin^2\tilde\theta}{\kappa_{\min} W}\frac{h^+ \, |\tan \tilde \theta|^2}{ 2\kappa_{\min} W}\\
			&\ll \frac{h^+}{2\kappa_{\min}^2} \|f\|_\infty \sup_{\theta} \frac{W^2}{\sin\theta}
			\sum_{\xi = W \max\{ \frac{1}{\eta}, \frac{1}{\sqrt{C_\kappa \sin \theta}} \} }^{\frac{W}{(C_\kappa \sin \theta)^2}}\frac{1}{\xi^4}
			\le \frac{h^+}{3\kappa_{\min}^2 C_\kappa^{3/2}} \|f\|_\infty \frac{\sqrt{\sin\theta}}{W} < \frac18,
		\end{align*}
		provided we take $W\geq 8h^+/(3\kappa_{\min}^2 C_\kappa^{3/2})$.\\
		
		\textbf{Estimating $\Var\left(Pf\Big|_{\{\sin \tilde\theta < \eta\}}\right)$.}
		For this part we proceed as in Section~\ref{subsec:varsmall} and we only sketch the argument. Here the calculations
		are easier due to the absence of $e^{itX}-1$.
		
		In short, \eqref{eq:varsum} and~\eqref{eq:var1} are replaced by
		\begin{align}\label{eq:varsum111}
			\Var_{\theta}(Pf|_{\{ \sin \tilde \theta < \eta\} } )
			\leq
			\sum_{|\xi| \geq \xi_\eta} \sum_{\ell \in \{ L,R,D\}} \frac{C_p W}{\xi^2}
			\max_{R_i \in p(\xi,\ell)} \Var_{\theta} \left( 
			\frac{f(\tilde \theta)}{|\Psi'_{R_i}(\tilde\theta)|} \frac{\sin \tilde \theta}{\sin \theta}  \1_{J_{\xi,\ell}}(\tilde\theta)  \right)
		\end{align}
		with
		\begin{align}\label{eq:var11111}
			\Var_{\theta} & \left(
			\frac{f(\tilde \theta)}{|\Psi'_{R_i}(\tilde\theta)|} \frac{\sin \tilde \theta}{\sin \theta} \1_{J_{\xi,\ell}}(\tilde\theta)  \right)
			=
			\Var_{\theta} \left( \frac{1}{\sin\theta} |_{\Psi_{R_i}(J_{\xi,\ell})} \right) \sup_{\tilde\theta \in J_{\xi,\ell}}
			\left| \frac{f(\tilde \theta)}{|\Psi'_{R_i}(\tilde\theta)|} \sin \tilde \theta \,  \1_{J_{\xi,\ell}}(\tilde\theta) \right| \nonumber \\
			& \quad + \ \sup_{\theta \in \Psi_{R_i}(J_{\xi,\ell})}\frac{1}{\sin\theta}
			\Var_{\tilde \theta}
			\left( \frac{f(\tilde \theta)}{|\Psi'_{R_i}(\tilde\theta)|} \sin \tilde \theta \,  \1_{J_{\xi,\ell}}(\tilde\theta) \right)\nonumber\\
			& =\, V_1(\xi,f,W)+V_2(\xi,f,W).
		\end{align}

		For $V_1$ we just need to recall that
		$ \Var_{\theta}\left(\frac{1}{\sin\theta}|_{\Psi_{R_i}(J_{\xi,\ell})} \right) \leq \sup_{\theta \in \Psi_{R_i}(J_{\xi,\ell})} \frac{1}{\sin\theta}$ by the proof of Lemma~\ref{lem:continf} (first lines below~\eqref{eq:var1}).
		Hence, the sums over terms with $V_1(\theta, f,W)$ can be dealt with similarly to estimating $\|Pf \Big|_{\{\sin \tilde\theta < \eta\}}\|_{\infty}$, which gives another term strictly less than $1/8$, or similarly to estimating $I$ inside the proof of Lemma~\ref{lem:continf}.

		For $V_2$, we proceed as in estimating $II$ inside the proof of Lemma~\ref{lem:continf}.
		The absence of the factor $e^{itX}-1$ much simplifies the calculation. More precisely,
		
		\begin{eqnarray}\label{eq:II111}
			V_2(\xi,f,W) &\leq& \sup_\theta \left(
			\frac{\Var(f)}{\sin\theta}
			\left| \frac{f(\tilde\theta) \, \sin \tilde\theta}{|\Psi'_{R_i}| } \right| + \
			\frac{\| f \|_\infty}{\sin\theta} \Var_{\tilde\theta} \left( \frac{\sin \tilde\theta}{|\Psi'_{R_i}| } \1_{J_{\xi,\ell}} \right) \right).
		\end{eqnarray}
		From Corollary~\ref{cor:Q} we have $\Var_{\tilde\theta} \left( \frac{\sin \tilde\theta}{|\Psi'_{R_i}| } \1_{I_{\xi,\ell}} \right) \leq C |\xi|^{-5/2}$ and for left cheek collisions, we have $|\Psi'_{R_i}(\tilde \theta)| \leq \frac{\sin^2 \tilde\theta}{W \kappa_{\min}}$ and $\sin \theta \geq \sin^2 \tilde \theta/C_\kappa^2 \sim W^2/(C_\kappa \xi)^2$.
		Therefore
		$$
		V_2(\xi,f,W) \leq \frac{C_\kappa^2 }{\kappa_{\min} W |\xi|} \Var(f) +  \frac{C_\kappa^2}{W^2 |\xi|^{\frac12}} \| f \|_\infty.
		$$
		Taking the sum over all relevant $(\xi,\ell)$ and recalling that $\xi_\eta \sim W/\eta$,
		we find
		\begin{eqnarray*}
			\sum_{|\xi| > \xi_\eta} \sum_{\ell \in \{L,R,D\}} \frac{C_p W}{\xi^2} \max_{R_i \in p(\xi,\ell)}
			V_2(\xi,f,W)
			&\leq& 3C_p\sum_{|\xi| > \xi_\eta}
			\frac{C_\kappa^2 }{\kappa_{\min} |\xi|^3} \Var(f) +  \frac{C_\kappa^2}{W |\xi|^{\frac52}} \| f \|_\infty \\
			&\leq&  \frac{3 C_p C_\kappa^2 \eta^2}{\kappa_{\min}W^2}  \| f \|_{BV} < \frac18,
		\end{eqnarray*}
		for $W > 2C_\kappa\eta \sqrt{6C_p/\kappa_{\min}}$.
		The conclusion follows by adding the sums over $V_1$ and $V_2$.
	\end{proof}

	\subsection{Estimating $\|Pf\|_{BV}$ when $\sin\theta\ge\eta$}
	\label{subsec:lyaw}

	In this section we proceed as in Section~\ref{subsec:contaw}
	without the presence of of the displacement $X$.

	\begin{lemma}\label{lem:lylarge}
		Let $\eta > 0$ be as in Lemma~\ref{lem:Lambda}.
		There exists $\alpha\in (0,1)$ and $C, C'>0$ so that
		\begin{equation*}
			\Var_\theta\left( P_{R_i} f\big|_{\{\sin \tilde\theta \geq \eta\} } \right)
			\leq \alpha \Var(f) + C \| f \|_\infty,\quad \left\|\left( P_{R_i} f\big|_{\{\sin \tilde\theta \geq \eta\} } \right)\right\|_\infty
			\leq C' \| f \|_\infty.
		\end{equation*}
	\end{lemma}
	
	\begin{proof}If $\theta \mapsto \Psi_{R_i}^{-1}(\theta) = \tilde \theta \in J_{\xi,\ell}$
		refers to a single monotone branch of $\Psi_{R_i}^{-1}$,
		then $\Var_\theta (f \circ \Psi_{R_i}^{-1}) = \Var_{ \theta}(f|_{I_{\xi,\ell}})$. Recall from Lemma~\ref{lem:SV} that $\sin \tilde \theta \geq \eta$ implies that $\sin \theta \geq \eta^2/ C_\kappa^2$.
		Compute that
		\begin{eqnarray*}
			\Var_\theta\left( P_{R_i} f\big|_{\{\sin \tilde\theta \geq \eta\} } \right) &=&
			\Var_\theta\left( \frac{1}{\sin \theta} \sum_{\ell \in \Lambda_\eta}
			\frac{f \circ \Psi^{-1}_{R_i}(\theta) \cdot \sin(\Psi_{R_i}^{-1}(\theta))}
			{|\Psi'_{R_i}(\Psi_{R_i}^{-1}(\theta)) } \1_{J_{\xi,\ell}}(\Psi_{R_i}(\theta))\right) \\
			&\leq&
			\left\| \frac{1}{\sin \theta}\big|_{\{ \sin \theta \geq \eta \}} \right\|_\infty \sum_{\ell \in \Lambda_\eta}
			\Var_{ \theta}\left(  \frac{f(\tilde\theta) \cdot \sin(\tilde\theta)}
			{|\Psi'_{R_i}(\tilde \theta) | } \1_{J_{\xi,\ell}}(\tilde \theta) \right) \\
			&& +\, \Var\left(\frac{1}{\sin \theta} \big|_{\sin \theta \geq \eta} \right)
			\left\| \sum_{\ell \in \Lambda_\eta}
			\frac{ f(\tilde\theta) \cdot \sin(\tilde\theta)} {|\Psi'_R(\tilde \theta) | } \1_{J_{\xi,\ell}}(\tilde \theta) \right\|_\infty \\
			&\leq& \frac{C_\kappa^2}{\eta^2} \sum_{\ell \in \Lambda_\eta} \left(
			\Var_{ \theta}(f|_{J_{\xi,\ell}}) \left\| \frac{\sin(\tilde\theta)}
			{|\Psi'_{R_i}(\tilde \theta) | } \1_{J_{\xi,\ell}}(\tilde \theta)\right\|_\infty \right. \\
			&& \left. \qquad \qquad \qquad + \, \| f|_{J_{\xi,\ell}} \|_\infty \Var_{ \theta}\left( \frac{\sin(\tilde\theta)}
			{|\Psi'_{R_i`}(\tilde \theta) | } \1_{J_{\xi,\ell}}(\tilde \theta)\right) \right) \\
			&&  + \, \frac{2C_\kappa^2}{\eta^2} \| f \|_\infty \sum_{\ell \in \Lambda_\eta}
			\left\| \frac{\sin(\tilde\theta) }{ |\Psi'_{R_i`}(\tilde \theta) | }\1_{J_{\xi,\ell}}(\tilde \theta) \right\|_\infty \\
			&\leq& \frac{C}{2\eta^2 W}\sum_{\ell \in \Lambda_\eta}
			\Var(f|_{J_{\xi,\ell}}) + \frac{2N W}{\eta^3} \frac{C}{W} \| f \|_\infty + \frac{N W}{\eta^3} \frac{C}{W} \| f \|_\infty\\
			&\leq& \frac{C}{2\eta^2 W} \Var(f) + \frac{3NC}{2\eta^3} \| f \|_\infty.
		\end{eqnarray*}
		By taking $W > 3C/(2\eta^2)$, this gives a Lasota-Yorke inequality for this part of the variation with $\alpha = \frac13$, already for the non-averaged transfer operator. Averaging cannot undo this, so we have
		$\Var_\theta\left( P f\big|_{\{\sin \tilde\theta \geq \eta\} } \right) \leq \frac13 \Var(f) + C' \| f \|_\infty$ for some $C'>0$. This proves the statement on the variation. The estimate for the infinity norm is simpler and omitted.
	\end{proof}
	
	\subsection{Proof of Proposition~\ref{prop:ly}}
	\label{subsec:ly}
	
	By Lemmas~\ref{lem:varpfsm} and~\ref{lem:lylarge},
	$\Var_\theta\left( P f\right) \leq \alpha \Var(f) + C \| f \|_\infty$,
	and $\left\| P f\right\|_\infty \leq C_2 \| f \|_\infty$, for some $C, C_2>0$.
	Repeated applications of these inequalities gives that
	$\Var_\theta\left( P^n f\right) \leq \alpha^n \Var(f) + C_0\| f \|_\infty$
	for some $C_0>0$. Therefore
	$\left\| P^n f\right\|_{BV} \leq \alpha^n \Var(f) + C_1\| f \|_\infty
	<\alpha^n \|f\|_{BV} + C_1\| f \|_\infty$, for some $C_1>0$, as desired.

	\section{Proof of Theorem~\ref{thm:main}}\label{sec:CLT}
	
	Recall that $X(\theta) = W/\tan\theta$, that $\E_\mu(X)=0$ by the symmetry of $d\mu = \frac12 \sin \theta d\theta$,
	and that
	$S_nX := \sum_{i=0}^{n-1} X_i$ where $X_i = X \circ (\Psi_{R_i} \circ \dots \circ \Psi_{R_0})$.
	
	In the rest of the section we show that for all $t\in\R$, as $n\to\infty$
	\begin{equation}\label{eq:F}
		\E_{\nu^{\otimes\Z}\times\mu}\left(e^{it\frac{S_nX}{\sqrt{n\log n}}}\right)\to e^{-\frac{W^2}{2} t^2}.
	\end{equation}
	Provided~\eqref{eq:F} holds, the conclusion of Theorem~\ref{thm:main} follows by the Levy Continuity Theorem. This means that the Gaussian random variable has mean $0$ and variance $W$.
	
	We need to study the RHS of~\eqref{eq:F} relating to the behaviour of $P_t$, which is Nagaev's method.
	By, for instance, repeating word by word the argument used in the proof of~\cite[Lemma 3.7]{ANS}, we obtain
	\begin{equation}\label{eq:F2}
		\E_{\nu^{\otimes\Z}\times\mu}\left(e^{it\frac{S_nX}{\sqrt{n\log n}}} f\right)
		=\int_{(0,\pi)}P_{\frac{t}{\sqrt{n\log n}}}^n f\, d\mu.
	\end{equation}

	We record the following easy lemma, which shows that the displacement
	$X(\theta) = W/\tan(\theta)$ barely fails to be in $L^2(\mu)$.
	
	\begin{lemma}\label{lem:tailX}
		The measures $\mu(\theta: X(\theta)>N)=\frac{W^2}{4N^2}(1+o(1))$ and $\mu(\theta: X(\theta)<-N)=\frac{W^2}{4N^2}(1+o(1))$ as $N\to\infty$.
	\end{lemma}
	
	\begin{proof}
		For $N> 0$ given, it holds that $X(\theta) = W/\tan\theta>N$ if and only if $0 < \theta < \arctan(W/N)$. Therefore
		$$
		\mu(\{\theta: X(\theta) > N \}) =  \int_{0}^{\pi}\1_{(N, \infty)}(X(\theta) )\, d\mu 
		=  \int_{0}^{\arctan(W/N)} \frac12 \sin \theta d\theta
		= \frac12 \left( 1-\cos\arctan\left(\frac{W}{N}\right)\right).
		$$
		Using a Taylor approximation as $N \rightarrow \infty$, we find that
		$$
		\cos\arctan\left(\frac{W}{N}\right) = \frac{1}{\sqrt{1+\frac{W^2}{N^2}}} = 1 - \frac{W^2}{2N^2}(1+o(1)),
		$$
		so $\mu(\{\theta: X(\theta)>N\}) = \frac{W^2}{4N^2}(1+o(1))$. The statement on $\mu(\{\theta: X(\theta)<-N\})$ follows likewise.
	\end{proof}

	\subsection{Spectral decomposition for $P_t$}
	
	Repeating the steps of the proof of Proposition~\ref{prop:ly},
	we can show that the average perturbed operator $P_t$, $t\in\R$, also satisfies
	the inequalities
	$\|P_t^nf\|_{BV} \le \alpha^n\|f\|_{BV}+C\|f\|_\infty$
	and $\|P_t^nf\|_\infty \le C'\|f\|_\infty$
	for some $\alpha\in (0,1)$ and $C,C'>0$.
	By Proposition~\ref{prop:cont}, the family $(P_t), t\in\R$, is
	continuous, when regarded as operators acting on $BV$.
	As a consequence, the associated eigenfamilies are also continuous in $t$.
	That is, the family of dominating eigenvalues
	$\lambda_t$ with corresponding eigenprojector operators $\Pi_t$
	and eigenvectors $v_t$ is continuous in $t$ (with the same continuity bound as that of Proposition~\ref{prop:cont}.)
	The family of eigenvalues $\lambda_t$ is well-defined for $t \in B_\eps(0)$.

	Recall the spectral decomposition for $P$ in
	Section~\ref{subsec:spde}. Standard arguments for smooth perturbation of linear operators (see~\cite{AD01, HL})
	ensure that for all $|t|<\eps$,
	\begin{align}\label{perdec}
		P_t^n=\lambda_t^n\Pi_t+ Q_t^n,\quad \Pi_t Q_t=Q_t\Pi_t=0 \quad \text{ and } \quad \|Q_t^n f\|_{BV}\le \delta^n \quad \text{ for some }\delta\in (0,1).
	\end{align}
	
	\subsection{Asymptotics of the dominating eigenvalue $\lambda_t$}
	\label{subsec:ev}

	Using Lemma~\ref{lem:tailX} and Proposition~\ref{prop:cont}
	we obtain the asymptotics of $\lambda_t$.
	
	\begin{lemma}\label{lem:lam}
		$1-\lambda_t = \frac{W^2}{2} t^2\log(1/|t|)(1+o(1))$ as $t\to 0$.
	\end{lemma}
	
	\begin{proof}
		Write $v_t$ for the eigenfunction of $\lambda_t$. Note that
		\begin{align*}
			1-\lambda_t=\int (1-e^{itX})\, v_0\,d\mu+\int (1-e^{itX})(v_t- v_0)\, d\mu.
		\end{align*}
		Here $v_0\equiv 1$ is the eigenvector associated with the eigenvalue $\lambda_0=1$. From here onward the argument is standard, see~\cite{ADb}. In particular, the first part of the calculations used in~\cite[Proof of Theorem 3.1]{ADb} shows that
		$$
		\int (1-e^{itX})\, v_0\, d\mu =\int (1-e^{itX}+itX)\, v_0\, d\mu
		= \frac{L(1/|t|)}{2} t^2  (1+o(1)),
		$$
		for $L(1/t) := \int_{-1/|t|}^{1/|t|} u^2 \, d\mathbb P(u)$.
		Here $\mathbb P(u) = \mu(\theta : X(\theta) < u)$, 	so the tail estimates of Lemma~\ref{lem:tailX} and integration by parts
		give
		\begin{eqnarray*}
			L(1/|t|) &=& \int_{-1/|t|}^{1/|t|} u^2 \, d\mathbb P(u)
			= \int_{-1/|t|}^0 u^2 \, d\mathbb P(u) - \int_0^{1/|t|} u^2 \, d(1-\mathbb P(u)) \\
			&=& -\int_{-1/|t|}^0 2u \mathbb P(u) \, du +
			\left[ u^2 \mathbb P(u)\right]_{-1/|t|}^0
			+ \int_0^{1/|t|}  2u (1-\mathbb P(u)) \, du -
			\left[ u^2 (1-\mathbb P(u))\right]_0^{1/|t|} \\
			&\sim& -\int_{-1/|t|}^0 2u \max\left\{ \frac{W^2}{4u^2}, \frac12\right\} \, du - \frac{W^2}{4}
			+ \int_0^{1/|t|}  2u \max\left\{ \frac{W^2}{4} , \frac12\right\}\, du - \frac{W^2}{4} \\
			&\sim& W^2 \log(1/|t|) + O(W^2) = W^2 \log(1/|t|) (1+o(1)),
		\end{eqnarray*}
		where $\sim$ indicates factors $1+o(1)$ as $t \to 0$.
		By Proposition~\ref{prop:cont} and standard perturbation theory of linear operators, $\|v_t-v_0\|_{BV}=O(|t|)$. Since $BV\subset L^\infty$,
		\begin{align*}
			\left|\int (1-e^{itX})(v_t- v_0)\, d\mu\right|
			\le |t|\, \|v_t-v_0\|_{\infty}\int |X|\, d\mu\ll |t|\, \|v_t-v_0\|_{BV}
			\ll t^2.
		\end{align*}
		and the conclusion follows.
	\end{proof}
	
	\subsection{Proof of Equation~\eqref{eq:F}}
	\label{subsec:ev2}
	
	By Lemma~\ref{lem:lam},
	$\lambda_t^n= e^{-n\frac{W^2}{2} t^2\log(1/|t|)(1+o(1))}$, as $t\to 0$.
	By Proposition~\ref{prop:cont}, $\|\Pi_t-\Pi_0\|_{BV}=O(|t|)$.
	Combining this with~\eqref{perdec}, we get
	$P_t^n f= e^{-n \frac{W^2}{2} t^2\log(1/|t|)(1+o(1))}\int f\, d\mu(1+o(1))$.
	
	Recalling~\eqref{eq:F2}, we get that for any $t\in\R$, as $n\to\infty$,
	$$
	\E_{\nu^{\otimes\Z}\times\mu}\left(e^{it\frac{S_nX}{\sqrt{n\log n}}} f\right)=e^{-\frac{n}{n\log n}\frac{W^2}{2} t^2\log(n/|t|)(1+o(1))}\int f\, d\mu(1+o(1))\to e^{-\frac{W^2}{2} t^2}\int f\, d\mu,
	$$
	as required.
	\\[5mm]
	{\bf Declaration:} The authors have no relevant financial or non-financial interests to disclose. There are no relevant research data other than presented in this paper.

\end{document}